\documentclass[english,11pt]{article}
\usepackage{tikz}
\usetikzlibrary{decorations.pathreplacing}
\usepackage{enumerate}
\usepackage[shortlabels]{enumitem}
\usepackage{verbatim}
\usepackage[margin=1in]{geometry}
\usepackage{amssymb}
\usepackage{caption}
\usepackage{bbm}
\usepackage{amsthm}
\usepackage{enumitem}

\usepackage{scalerel}    
\usepackage{stmaryrd}   

\makeatletter
\DeclareRobustCommand\widecheck[1]{{\mathpalette\@widecheck{#1}}}
\def\@widecheck#1#2{%
    \setbox\z@\hbox{\m@th$#1#2$}%
    \setbox\tw@\hbox{\m@th$#1%
       \widehat{%
          \vrule\@width\z@\@height\ht\z@
          \vrule\@height\z@\@width\wd\z@}$}%
    \dp\tw@-\ht\z@
    \@tempdima\ht\z@ \advance\@tempdima2\ht\tw@ \divide\@tempdima\thr@@
    \setbox\tw@\hbox{%
       \raise\@tempdima\hbox{\scalebox{1}[-1]{\lower\@tempdima\box
\tw@}}}%
    {\ooalign{\box\tw@ \cr \box\z@}}}
\makeatother

\usepackage{rotating,centernot,cancel}    

\usepackage{amsmath}
\usepackage{tocloft}
\usepackage{float}

\usepackage{tcolorbox}   
\PassOptionsToPackage{hyphens}{url}\usepackage{hyperref}
\usepackage{babel}
\theoremstyle{plain}

\input amssym.def
\input amssym.tex

\def\beq{\begin{equation}}
\def\eeq{\end{equation}}
\def\beqn{\begin{eqnarray}}
\def\eeqn{\end{eqnarray}}

\usepackage{mathtools}
\DeclarePairedDelimiter\floor{\lfloor}{\rfloor}

\newtheorem{theorem}{Theorem}[section]

\newtheorem*{lemma*}{Lemma}
\newtheorem{proposition}[theorem]{Proposition} 

\theoremstyle{remark}
\newtheorem{remark}[theorem]{Remark}

\theoremstyle{definition}

\numberwithin{figure}{section}

\def\R{\mathbb R}

\def\Z{\mathbb Z}

\def\to{\rightarrow}
\def\dlim[#1][#2]{\lim_{#1 \to #2, #1 \neq #2}}

\def\Var{\text{$\mathbb{V}$ar}}
\def\Exp{\textup{Exp}}

\newcommand{\be}{\begin{equation}}
\newcommand{\ee}{\end{equation}}

\newcommand\bbullet{{{\scaleobj{0.6}{\bullet}}}} 
\newcommand\mydots{\hbox to 1em{.\hss.\hss.}}

\def\wt{\widetilde}    \def\wc{\widecheck}

\def\coal{{\mathbf z}} 

\def\bgeod#1#2{\mathbf{b}^{#1, #2}}

\newcommand{\myfootnote}[1]{
    \renewcommand{\thefootnote}{}
    \footnotetext{\scriptsize#1}
    \renewcommand{\thefootnote}{\arabic{footnote}}
}

\allowdisplaybreaks

\title{Independence property of the Busemann function \\ in exactly solvable KPZ models}
\author{Xiao Shen\thanks{\scriptsize{Department of Mathematics, University of Utah, Utah, USA. \texttt{xiao.shen@utah.edu}}}}
\date{}
\begin{document}
\maketitle

\begin{abstract}
The study of Kadar-Parsi-Zhang (KPZ) universality class has been a subject of great interest among mathematicians and physicists over the past three decades. A notably successful approach for analyzing KPZ models is the coupling method, which hinges on understanding random growth from stationary initial conditions defined by Busemann functions. To advance in this direction, 
we investigate the independence property of the Busemann function across multiple directions in various exactly solvable KPZ models. These models encompass the corner growth model, the inverse-gamma polymer, Brownian last-passage percolation, the O'Connell-Yor polymer, the KPZ equation, and the directed landscape.
In the context of the corner growth model, our result states that disjoint Busemann increments in different directions along a down-right path are independent, as long as their associated semi-infinite geodesics have nonempty intersections almost surely. The proof for the independence utilizes the queueing representation of the Busemann process developed by Sepp\"al\"ainen et al. As an application, our independence result yields a near-optimal probability upper bound (missing by a logarithmic factor) for the rare event where the endpoint of a point-to-line inverse-gamma polymer is close to the diagonal.

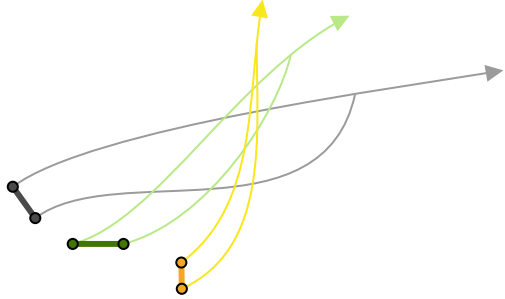
\begin{figure}[h]
\begin{center}
\tikzset{every picture/.style={line width=0.75pt}} 

\begin{tikzpicture}[x=0.75pt,y=0.75pt,yscale=-1,xscale=1]

\draw [color={rgb, 255:red, 155; green, 155; blue, 155 }  ,draw opacity=1 ]   (237.27,101.37) .. controls (276.67,71.82) and (432.72,51.26) .. (482.5,43) ;
\draw [shift={(484.7,42.63)}, rotate = 170.38] [fill={rgb, 255:red, 155; green, 155; blue, 155 }  ,fill opacity=1 ][line width=0.08]  [draw opacity=0] (8.93,-4.29) -- (0,0) -- (8.93,4.29) -- cycle    ;
\draw [color={rgb, 255:red, 155; green, 155; blue, 155 }  ,draw opacity=1 ]   (248.6,117.2) .. controls (288.6,87.2) and (393.9,128.63) .. (409.9,54.63) ;
\draw [color={rgb, 255:red, 184; green, 233; blue, 134 }  ,draw opacity=1 ]   (267.5,130.23) .. controls (309.76,119.1) and (349.88,45.3) .. (404.59,16.33) ;
\draw [shift={(407.1,15.03)}, rotate = 153.43] [fill={rgb, 255:red, 184; green, 233; blue, 134 }  ,fill opacity=1 ][line width=0.08]  [draw opacity=0] (8.93,-4.29) -- (0,0) -- (8.93,4.29) -- cycle    ;
\draw [color={rgb, 255:red, 184; green, 233; blue, 134 }  ,draw opacity=1 ]   (293.1,130.23) .. controls (336,118.93) and (371.1,63.83) .. (377.5,34.63) ;
\draw [color={rgb, 255:red, 248; green, 231; blue, 28 }  ,draw opacity=1 ]   (322.27,139.63) .. controls (361.47,110.22) and (355.15,56.67) .. (362.79,9.74) ;
\draw [shift={(363.27,6.88)}, rotate = 100.01] [fill={rgb, 255:red, 248; green, 231; blue, 28 }  ,fill opacity=1 ][line width=0.08]  [draw opacity=0] (8.93,-4.29) -- (0,0) -- (8.93,4.29) -- cycle    ;
\draw [color={rgb, 255:red, 248; green, 231; blue, 28 }  ,draw opacity=1 ]   (322.52,152.88) .. controls (367.27,132.13) and (360.27,75.03) .. (360.27,28.63) ;
\draw [color={rgb, 255:red, 74; green, 74; blue, 74 }  ,draw opacity=1 ][line width=2.25]    (237.27,101.37) -- (248.6,117.2) ;
\draw  [fill={rgb, 255:red, 74; green, 74; blue, 74 }  ,fill opacity=1 ] (234.72,101.37) .. controls (234.72,99.96) and (235.86,98.82) .. (237.27,98.82) .. controls (238.68,98.82) and (239.82,99.96) .. (239.82,101.37) .. controls (239.82,102.78) and (238.68,103.92) .. (237.27,103.92) .. controls (235.86,103.92) and (234.72,102.78) .. (234.72,101.37) -- cycle ;
\draw  [fill={rgb, 255:red, 74; green, 74; blue, 74 }  ,fill opacity=1 ] (246.05,117.2) .. controls (246.05,115.79) and (247.19,114.65) .. (248.6,114.65) .. controls (250.01,114.65) and (251.15,115.79) .. (251.15,117.2) .. controls (251.15,118.61) and (250.01,119.75) .. (248.6,119.75) .. controls (247.19,119.75) and (246.05,118.61) .. (246.05,117.2) -- cycle ;
\draw [color={rgb, 255:red, 65; green, 117; blue, 5 }  ,draw opacity=1 ][line width=2.25]    (267.5,130.23) -- (293.1,130.23) ;
\draw  [fill={rgb, 255:red, 65; green, 117; blue, 5 }  ,fill opacity=1 ] (264.95,130.23) .. controls (264.95,128.83) and (266.09,127.68) .. (267.5,127.68) .. controls (268.91,127.68) and (270.05,128.83) .. (270.05,130.23) .. controls (270.05,131.64) and (268.91,132.78) .. (267.5,132.78) .. controls (266.09,132.78) and (264.95,131.64) .. (264.95,130.23) -- cycle ;
\draw  [fill={rgb, 255:red, 65; green, 117; blue, 5 }  ,fill opacity=1 ] (290.55,130.23) .. controls (290.55,128.83) and (291.69,127.68) .. (293.1,127.68) .. controls (294.51,127.68) and (295.65,128.83) .. (295.65,130.23) .. controls (295.65,131.64) and (294.51,132.78) .. (293.1,132.78) .. controls (291.69,132.78) and (290.55,131.64) .. (290.55,130.23) -- cycle ;
\draw [color={rgb, 255:red, 245; green, 166; blue, 35 }  ,draw opacity=1 ][fill={rgb, 255:red, 255; green, 255; blue, 255 }  ,fill opacity=1 ][line width=2.25]    (322.27,139.63) -- (322.52,152.88) ;
\draw  [fill={rgb, 255:red, 245; green, 166; blue, 35 }  ,fill opacity=1 ] (319.72,139.63) .. controls (319.72,138.22) and (320.86,137.08) .. (322.27,137.08) .. controls (323.68,137.08) and (324.82,138.22) .. (324.82,139.63) .. controls (324.82,141.03) and (323.68,142.17) .. (322.27,142.17) .. controls (320.86,142.17) and (319.72,141.03) .. (319.72,139.63) -- cycle ;
\draw  [fill={rgb, 255:red, 245; green, 166; blue, 35 }  ,fill opacity=1 ] (319.97,152.88) .. controls (319.97,151.47) and (321.11,150.33) .. (322.52,150.33) .. controls (323.93,150.33) and (325.07,151.47) .. (325.07,152.88) .. controls (325.07,154.28) and (323.93,155.42) .. (322.52,155.42) .. controls (321.11,155.42) and (319.97,154.28) .. (319.97,152.88) -- cycle ;
\end{tikzpicture}
\captionsetup{width=.89\linewidth}
\caption*{\small Figure: The three Busemann increments in three different directions along a down-right path are shown above. Due to the specific positions and directions of the Busemann increments, almost surely the semi-infinite geodesics will intersect each other. By our result, the three Busemann increments are independent.}
\end{center}
\end{figure}
\end{abstract}

\myfootnote{Date: \today}
\myfootnote{2010 Mathematics Subject Classification. 60K35, 	60K37}
\myfootnote{Key words: last passage percolation, directed polymer, Kardar-Parisi-Zhang, Busemann function, random growth model.}

\section{Introduction}\label{intro}

The exploration of universality is at the core of probability and statistics. It refers to a remarkable phenomenon observed in many random growths, where diverse systems with different underlying microscopic details exhibit universal macroscopic behavior. 
One classic example illustrating universality is the \textit{central limit theorem} (CLT), which states that the large-scale behavior of the centered sum becomes independent of the specific distribution of the individual summands. With the fluctuation exponent $1/2$ and the Gaussian distribution as the
scaling law, this is often referred to as the \textit{Gaussian universality class}.

Just as coin tosses led to the CLT, based on physical experiments, Kardar, Parisi, and Zhang (KPZ) predicted another strikingly different universality class in 1986 \cite{Kar-Par-Zha-86}. This new universal behavior was anticipated to emerge across a wide range of stochastic models with spatial dependence.
Extensive computer simulations and empirical laboratory experiments have collectively suggested that the \textit{KPZ universality class} is extremely rich. It encompasses percolation models, directed polymers, interacting particle systems, random tilings, stochastic partial differential equations, and more.
Hence, studying the KPZ universality offers a unifying framework that enables one to comprehend a wide range of phenomena across different probabilistic models.
However, demonstrating that a model belonging to the KPZ universality class, meaning proving a statement similar to the CLT,  has remained an exceedingly challenging task.

Although the questions of universality remain out of reach, there are several \textit{exactly solvable KPZ models} which can be rigorously analyzed. The term ``exactly solvable models" refers to instances where specific probability distributions are employed within the models, enabling precise calculations.
The idea of selecting specialized weights that facilitate calculational insights has been a recurring theme throughout history. A notable instance of this is found in a particular version of the CLT, which was initially established in the 1700s for the scenario where the summands consist of i.i.d.~Bernoulli random variables. By utilizing Stirling's formula, it became feasible to formulate the binomial density and then take the limit as $n$ grew large, ultimately yielding an explicit formula which is now known as the Gaussian distribution. This strategic choice of weights, adhering to the Bernoulli distribution, transformed the i.i.d. sum into an exactly solvable model. This historical progression underscores the significance of exploiting exactly solvable models to unveil deep probabilistic insights.

Over the last twenty-five years, several different methods have surfaced for investigating the exactly solvable KPZ models. These approaches have originated from a spectrum of mathematical disciplines, including representation theory, algebraic combinatorics, probability theory, and stochastic analysis. 
Following the seminal work of Baik, Deift and Johansson \cite{Bai-Dei-Joh-99}, a particularly successful approach is recognized as \textit{integrable probability}. It relies on the remarkable structure of the exactly solvable models which admits the usage of algebraic tools.  Consequently, it becomes possible to explicitly write down the formulas for the one-point or multi-point probability distribution functions of these models. Another related approach is built upon the resampling property observed in line ensembles, a concept introduced by Corwin and Hammond \cite{bgibbs}. This is a mix of integrable probability and classical probability, and it has led to one of the most fundamental and well-studied principles for KPZ models, known as local Brownianity of the weight profile.

However, for the vast non-solvable members of the KPZ class, these intricate algebraic methods are not directly applicable. With the eventual goal of extending results beyond the exactly solvable cases, several other lines of research were developed relying only on the broadly applicable probabilistic techniques and geometric arguments.  Notably, one such technique is known as the \textit{coupling method}. It involves coupling two evolutions of random growth, one with a specific initial condition and the other with a stationary initial condition, which is easier to analyze.  By controlling the differences between the two systems, one can
draw conclusions about the growth with the specific initial condition of interest. 

The general idea of coupling goes back to the seminal works of Griffeath,  Harris, Holley, Liggett and Spitzer from the 70's for studying the exclusion process \cite{MR397926, MR356292, MR268960, MR418291, MR375533}. The coupling technique closer to the context above was introduced by Sepp\"al\"ainen in the 90's in a series of seminal works on both solvable and non-solvable interacting particle systems \cite{Sep-96, Sep-97-aap-2, Sep-97-aap-1, Sep-98-mprf-2, Sep-98-aop-2, Sep-98-mprf-1, Sep-98-ptrf, Sep-99-aop}. These works aimed at proving shape theorems, establishing hydrodynamic limits, and deriving large deviation principles.
Following this, the method was employed in the influential work of Cator and Groeneboom \cite{Cat-Gro-06} where they verified the conjectured KPZ fluctuation for a stationary version of Hammersley's process. More recently, this effort has been revolutionized by Emrah, Janjigian and Sepp\"al\"ainen \cite{rtail1} that made certain quantitatively optimal bounds possible \cite{cgm_low_up, OC_tail, Lan-Sos-22-a-, asep_tail}.

One of the key inputs for utilizing the coupling method is to understand the random growth starting from the stationary initial condition, which is defined using the \textit{Busemann functions}. Our paper focuses on investigating the independence properties of Busemann functions within exactly solvable KPZ models. In a broader context, this contributes to the advancement of the coupling technique. To demonstrate the application of our results, we provide a near-optimal upper probability bound (missing only a logarithmic factor) for the rare event where the endpoint of a point-to-line polymer approaches the diagonal.

\vspace{3mm}
\noindent\textbf{Acknowledgements.} 
The author expresses gratitude to  Erik Bates, Chris Janjigian, Firas Rassoul-Agha, and Evan Sorensen for their insightful discussions regarding their works on the Busemann process. The author also extends sincere thanks to Timo Sepp{\"a}l{\"a}inen for the valuable suggestions provided for this paper. Furthermore, the author acknowledges partial support from the Wylie Research Fund at the University of Utah.

\section{Main results}

We begin by presenting our results for two representative KPZ lattice models: the corner growth model and the inverse-gamma polymer. These models represent the zero-temperature and positive-temperature KPZ models, respectively. This classification is due to the fact that, under certain scaling, the corner growth model can be viewed as a limit of the inverse-gamma polymer. Thus, the process of taking the limit is often referred to as scaling the temperature parameter from a positive value to zero.

To avoid introducing additional definitions and notation, we will defer the discussion of our results for the Brownian last-passage percolation (Brownian LPP), the O'Connell-Yor polymer, the KPZ equation, and the directed landscape to Section \ref{more_models} (see Theorem \ref{semi} and Theorem \ref{cont}).

\subsection{The corner growth model}
The \textit{corner growth model} (CGM) is defined on the integer lattice  $\mathbb{Z}^2$.  Let $\{\omega_{\mathbf z}\}_{\mathbf z\in \Z^2}$ be i.i.d.~Exp(1) random variables. 
The {\it last-passage value} $G_{\mathbf  x,\mathbf y}$ between two coordinatewise-ordered vertices $\mathbf x\le \mathbf y$ of $\Z^2$ is the maximal  total  weight along an  up-right nearest-neighbor path from $\mathbf x$ to $\mathbf y$: 
\begin{equation}\label{sec2G}G_{\mathbf x,\mathbf y} = \max_{\mathbf z_{\bbullet}\, \in\, \Pi^{\mathbf x,\mathbf y}} \sum_{k=0}^{|\mathbf y-\mathbf x|_1} \omega_{\mathbf z_k}
\end{equation}
where $\Pi^{\mathbf x,\mathbf y}$ is the collection of paths $\mathbf z_{\bbullet}=(\mathbf z_k)_{k=0}^{|\mathbf y-\mathbf x|_1}$ that satisfy $\mathbf z_0=\mathbf x$, $\mathbf  z_{|\mathbf y-\mathbf x|_1}=\mathbf y$, and $\mathbf z_{k+1}-\mathbf z_k\in\{\mathbf e_1, \mathbf e_2\}$.
The almost surely unique maximizing path is commonly referred to as the \textit{point-to-point geodesic} or \textit{finite geodesic}.

A semi-infinite up-right path $(\mathbf 
 z_i)_{i=0}^\infty$ is a \textit{semi-infinite geodesic} if any of its finite segments is a point-to-point geodesic, that is,  
$$\forall k<l \text{ in } \Z_{\geq0}:  \ (\mathbf z_i)_{i=k}^l \in \Pi^{\mathbf z_k, \mathbf z_l} \quad \text{ and } \quad  G_{\mathbf z_k, \mathbf z_l}= \sum_{i=k}^l \omega_{\mathbf z_i}.$$
For a direction $\boldsymbol\xi\in\R^2_{\ge0}\setminus\{(0,0)\}$, the semi-infinite path $(\mathbf z_i)_{i=0}^\infty$ is $\boldsymbol\xi$-\textit{directed} if $\mathbf z_i/|\mathbf z_i|_1 \rightarrow \boldsymbol \xi/|\boldsymbol\xi|_1$ as $i\rightarrow \infty$. 
In the CGM, it is natural to index 
 spatial directions $\boldsymbol\xi$  by a  real parameter  $\rho\in(0,1)$  through the equation  
\be\label{char1}
    \boldsymbol\xi[\rho] = \left((1-\rho)^2, \rho^2 \right) . 
\ee
Following the works \cite{multicoupier, buse3, fppcoal2, coalnew}, almost surely every semi-infinite geodesic has an asymptotic direction, and for each fixed direction $\boldsymbol\xi[\rho]$ where $\rho\in(0,1)$,   the following holds almost surely:
\begin{itemize}[noitemsep]
\item For each $\mathbf x\in\Z^2$ there is a unique $\boldsymbol\xi[\rho]$-directed semi-infinite geodesic $\bgeod{\,\rho}{\mathbf x} =  \left(\bgeod{\,\rho}{\mathbf x}_i\right)_{i=0}^\infty$ such that  $\bgeod{\,\rho}{\mathbf x}_0=\mathbf x$.  
\item For each pair $\mathbf x,\mathbf y\in\Z^2$, the semi-infinite geodesics in the same direction coalesce:  there is a coalescence point $\coal^\rho(\mathbf x,\mathbf y)$ such that 
$\bgeod{\,\rho}{\mathbf x} \cap \bgeod{\,\rho}{\mathbf y}  = \bgeod{\,\rho}{\coal^\rho(\mathbf x,\mathbf y)} $.
\end{itemize}

\subsection{The inverse-gamma polymer}\label{inv_intro}

Recall that a random variable $X$ has the inverse-gamma distribution with shape parameter $\mu\in(0,\infty)$, abbreviated as $X\sim \text{Ga}^{-1}(\mu)$, if $X^{-1}$ has the gamma distribution with the shape parameter $\mu$. The probability density function is given by  
$$f_{X}(x) = \frac{1}{\Gamma(\mu)} x^{-1-\mu}e^{-x^{-1}}\mathbbm{1}_{(0, \infty)}(x)$$
where $\Gamma(a) = \int_0^\infty s^{a-1}e^{-s} ds$ is the gamma function. 

The inverse-gamma polymer is also defined in $\mathbb{Z}^2$. Let $\{\omega_\mathbf z\}_{\mathbf z\in \mathbb{Z}^2}$ be i.i.d.~inverse-gamma distributed random variables with a fixed shape parameter $\mu \in (0, \infty)$. To maintain consistency in the notation with the CGM, we set $\mu=1$, although we note that the exact same arguments can also be extended for general $\mu\in \mathbb{R}_{>0}$. For two coordinatewise-ordered vertices $\mathbf x\leq \mathbf y$ of $\mathbb{Z}^2$, the point-to-point partition function is defined by 
\begin{equation}\label{def_part}
Z_{\mathbf x, \mathbf y} = \sum_{z_\bbullet \in \Pi^{\mathbf x, \mathbf y}} \prod_{k=0}^{|\mathbf y-\mathbf x|_1} \omega_{\mathbf z_k}.
\end{equation}
We use the convention $Z_{\mathbf x,\mathbf y} =  0$ if $\mathbf x\leq \mathbf y$ fails. 
Moreover, the free energy is defined as $\log Z_{\mathbf x, \mathbf y}$.

The \textit{quenched polymer measure} is a probability measure on the set $\Pi^{\mathbf x,\mathbf y}$ and is defined by 
$$Q_{\mathbf x,\mathbf y}\{\mathbf z_{\bbullet}\} = \frac{1}{Z_{\mathbf x,\mathbf y}} \prod_{i=0}^{|\mathbf x-\mathbf y|_1} \omega_{\mathbf z_i}.$$
Compared to the corner growth model, the finite geodesic can be thought of as a (random) Dirac-delta measure on $\Pi^{\mathbf x, \mathbf y}$ which is supported on the maximizing path of the last-passage value $G_{\mathbf x, \mathbf y}$, while $Q_{\mathbf x, \mathbf y}$ is a (random) probability measure on $\Pi^{\mathbf x, \mathbf y}$ which is concentrated around the maximizing path of the free energy $\log Z_{\mathbf x, \mathbf y}$.
Lastly, for $\rho \in (0, 1)$, the characteristic direction is defined by
$$\boldsymbol\xi[{{\rho}}] =  ({\Psi_1({{\rho}})}, {\Psi_1({{1-\rho}})})$$
where $\Psi_1(z) = \frac{d^2}{dz^2} \log\Gamma(z)$ is the trigamma function.

\subsection{Busemann function}
The Busemann function is named after the German mathematician Herbert Busemann, who introduced this concept in the context of
(nonrandom) metric geometry. Then the idea was brought to the study of semi-infinite geodesics in a closely related non-solvable KPZ model called first-passage percolation, in the seminal works \cite{MR2114988,Hof-08, New-95}. In the exactly solvable KPZ models, the Busemann functions have been an essential tool for understanding the random geometry and the space-time profiles through coupling techniques, which we briefly introduced in Section \ref{intro}.

In the context of the CGM,
for $\mathbf x, \mathbf y \in \mathbb{Z}^2$, the Busemann function in the direction $\boldsymbol \xi[\rho]$ is defined as 
\begin{equation}\label{def_buse}
B^\rho_{\mathbf x,\mathbf y} = G_{\mathbf x, \coal^\rho(x,y)} - G_{\mathbf y, \coal^\rho(x,y)}
\end{equation}
where $\coal^\rho(\mathbf x,\mathbf y)$ is the coalescent point of the semi-infinite geodesics $\bgeod{\rho}{\mathbf x}$ and  $\bgeod{\rho}{\mathbf y}$. Alternatively, it can also be seen as the following almost sure limit, where $\mathbf v_N$ is a sequence of vectors such that $|\mathbf v_N|_1 \to \infty$ and $\mathbf v_N/|\mathbf v_N|_1 \to \boldsymbol\xi[\rho]/|\boldsymbol\xi[\rho]|_1$,
$$ B^\rho_{\mathbf x,\mathbf y} = \lim_{N\to \infty}G_{\mathbf x, \mathbf v_N} - G_{\mathbf y, \mathbf v_N}$$
The second definition is useful in the positive-temperature setting, as there are no geodesics.  Thus, the Busemann function in the inverse-gamma polymer is defined as the limit of the difference between two free energies:
$$B^\rho_{\mathbf x, \mathbf y} = \lim_{N\to \infty} \log Z_{\mathbf x, \mathbf v_N} -  \log Z_{\mathbf y, \mathbf v_N}.$$
We note that the existence of these limits is non-trivial, and proofs can be found in \cite{CGMlecture} for the CGM and \cite{Geo-etal-15} for the inverse-gamma polymer.

For a fixed direction,  one significance of the Busemann function is that it establishes a coupling between the semi-infinite geodesic in the CGM and the geodesics in CGM starting with stationary initial conditions, also known as the (increment) stationary last-passage percolation. In the positive-temperature setting, it connects the semi-infinite polymer measure in the inverse-gamma polymer to the (ratio) stationary inverse-gamma polymer. These connections enable us to study various important geometric properties in these models. For instance, it allows us to quantify the coalescence speed of two semi-infinite geodesics or the sampled paths from two semi-infinite polymer measures from two different starting points \cite{dual, ras-sep-she-, seppcoal}. 

The Busemann function can be regarded as a process that evolves as the directional parameter $\rho$ varies from $0$ to $1$. Based on the intertwining argument
that was originally developed for TASEP in \cite{stat_tasep} by Ferrari and Martin, the joint distribution in the directional parameter was first studied by Fan and Sepp\"al\"ainen \cite{Fan-Sep-20} in the CGM. Their results have led to significant advancements in the studies of geodesics, such as the proof of the non-existence of bi-infinite geodesics using the coupling techniques \cite{balzs2019nonexistence}, the demonstration of the local stationarity between the finite and semi-infinite geodesics \cite{balzs2020local}, the establishment of the universality of the geodesic tree \cite{uni_tree}, and a complete description of the global geometry of the full set of semi-infinite geodesics \cite{Jan-Ras-Sep-22-}.
More recently, the description for the joint distribution of the Busemann function was then extended to  the inverse-gamma polymer \cite{bat-fan-sep-23-}, the Brownian LPP, the directed landscape  \cite{busani2023diffusive, 2021arXiv211210729S}, theO'Connell-Yor polymer and the KPZ equation \cite{groathouse2023jointly}. For related works that were built on these results, see \cite{Bus-Sep-22, 2022arXiv220313242B, busani2023scaling, 2021arXiv211210729S}.

Finally, we provide an additional list of references for the study of the Busemann function on general lattice models, such as directed polymers, first- and last-passage percolation with general i.i.d.~environments, see \cite{neg_buse, Dam-Han-14, Dam-Han-17,  Geo-Ras-Sep-17-ptrf-1, groathouse2023existence, Jan-Ras-20-aop}.

\subsection{Main results}\label{sec:main}

Fix $K \in \mathbb{Z}_{>0}$ and $\rho_1 < \mydots < \rho_K \in (0, 1)$. Define an infinite down-right path $\mathcal{Y} =  \{\mathbf y_i\}_{i\in \mathbb{Z}}$ with $\mathbf y_{i+1} - \mathbf y_i \in \{\mathbf e_1, -\mathbf e_2\}$. Fix any $i_1<i_2 < \mydots < i_{K-1}$, and define the collections of \textit{Busemann increments} (which are Busemann functions evaluated at different points) as
\begin{align}
\mathcal{B}^1 &= \{B^{\rho_1}_{\mathbf y_{i+1}, \mathbf y_i} : -\infty < i < i_1\}\nonumber\\
\mathcal{B}^k &= \{B^{\rho_k}_{\mathbf y_{i+1}, \mathbf y_i} : i_{k-1}\leq  i < i_{k} \} \quad \text{ for }k \in \{2, \mydots, K-1\}\label{def_Bk}\\
\mathcal{B}^K &= \{B^{\rho_K}_{\mathbf y_{i+1}, \mathbf y_i} : i_{K-1}\leq  i < \infty \}. \nonumber
\end{align}
Our first main result below establishes the independence of these increments for both the CGM and inverse-gamma polymer. 
\begin{theorem}\label{main}
The collection of random variables $\mathcal{B}^1\cup \mathcal{B}^2 \cup  \mydots \cup  \mathcal{B}^K$ are mutually independent. 
\end{theorem}

\begin{remark}
The marginal distributions for these Busemann increments are also known in the literature. Specifically, we have $B^{\rho_k}_{\mathbf x, \mathbf x+\mathbf e_1} \sim \Exp(1-\rho_k)$ and $B^{\rho_k}_{\mathbf x, \mathbf x+\mathbf e_2} \sim \Exp(\rho_k)$ for the CGM \cite{CGMlecture}, and ${B^{\rho_k}_{\mathbf x, \mathbf x+\mathbf e_1}} \sim \log \textup{Ga}^{-1}(1-\rho^k)$ and ${B^{\rho_k}_{\mathbf x, \mathbf x+\mathbf e_2}} \sim \log \textup{Ga}^{-1}(\rho^k)$ for the inverse-gamma polymer \cite{Geo-etal-15}.
\end{remark}

\begin{remark}
Similar independence results also hold for the Brownian LPP, the O'Connell-Yor polymer, the KPZ equation, and the directed landscape. We postpone the statement for these models to Sections \ref{more_models} (see Theorem \ref{semi} and Theorem \ref{cont}).
\end{remark}

Theorem \ref{main} generalizes two known results regarding the independence of Busemann increments. The first result corresponds to a specific scenario covered by Theorem \ref{main} when $K = 1$. In this case, it establishes the independence of Busemann increments along the same direction within a down-right path.
The second result relates to another specific case of Theorem \ref{main} where the down-right paths $\mathcal{Y}$ form a horizontal line and there are only two available directions, resulting in $K = 2$. This fact can be directly derived from the queue representation of the Busemann process, where independence can be interpreted as departure times before time $t$ being independent of the service times that occur after $t$.

As an application, our independence result yields a near-optimal probability upper bound (missing by a logarithmic factor) for the rare event where the endpoint of a point-to-line inverse-gamma polymer is close to the diagonal. 
Let $Q^{\textup{ptl}}_{\mathbf 0, N}$ denote the point-to-line quenched path measure on the collection of directed paths from $(0,0)$ to the anti-diagonal line $x+y = 2N$. And let $\{\textsf{end}\leq k\}$ denote the sub-collection of these paths that intersect the $\ell^\infty$ ball of radius $k$, centered at $(N, N)$. 
\begin{theorem} \label{mid}
There exist finite strictly positive constants $\delta_0, N_0, C$ such that whenever $0< \delta \leq \delta_0$ and $N\geq N_0$ 
\begin{equation}\mathbb{E}\Big[Q^{\textup{ptl}}_{\mathbf 0, N}\{\textup{\textsf{end}} \leq \delta N^{2/3}\} \Big] \leq C |\log {(\delta\vee N^{-2/3})}|^{10}(\delta\vee N^{-2/3}).
\end{equation}
\end{theorem}

Based on the small deviation work \cite{small_deviation_LPP} for the CGM, we anticipate the optimal bound to be $C(\delta \vee N^{-2/3})$. However, achieving this requires a crucial estimate for the number of disjoint geodesics inside a parallelogram, which was proved in \cite{noninf}. This particular argument relies on the notion of geodesics in an essential way, which makes it only applicable in the zero-temperature setting.
In contrast, in the positive-temperature setting, the best-known upper bound for the inverse-gamma polymer before our results is in the form of $C |\log {(\delta\vee N^{-2/3})}|^{10}\sqrt{\delta\vee N^{-2/3}}$, see Remark 2.13 from \cite{ras-sep-she-}.

\subsection{Organization of the paper}

The preliminaries required for the proofs are outlined in Section \ref{q_basic}. Subsequently, the proof of Theorem \ref{main} can be found in Section \ref{pf_main1}. Both sections present the arguments formulated within the context of the CGM.
To adapt these proofs for the inverse-gamma polymer, we will discuss the necessary modifications and present the proof of Theorem \ref{mid} in Section \ref{ig_poly}.
Finally, Section \ref{more_models} the discussion how these results can be extended to the Brownian LPP, the O'Connell-Yor polymer, the KPZ equation, and the directed landscape.

\section{Preliminaries}\label{q_basic}

For simplicity of notation, we will present the preliminaries in the context of the CGM. The modifications needed for the other models will be discussed in Section \ref{ig_poly} and Section \ref{more_models}.

We start by defining a \textit{last-passage percolation} (LPP) starting from a horizontal boundary. To define this model, without the loss of generality, let us suppose that the starting point of the last-passage percolation is the origin.
Let $\{Y_{\mathbf z}\}_{\mathbf z\in \mathbb{Z}^2}$ be a collection of arbitrary positive weights. The weights $\{Y_{(j,0)}\}_{j\in \mathbb{Z}}$ will be the boundary value, and denote them as $Y[0] = (Y_{(j, 0)})_{j\in \mathbb{Z}}$. Once $Y[0]$ is given, define $\{h_i\}_{i \in \mathbb{Z}}$ to be 
$$h_i = \begin{cases}
\sum_{j=0}^i Y_{(j,0)} \qquad & \textup{ if } i \geq 0\\
-\sum_{j=1}^{|i|} Y_{(-j,0)} \qquad & \textup{ if } i < 0.
\end{cases}
$$
Recall the bulk last-passage value $G$ defined in \eqref{sec2G}, then the last-passage value starting from the horizontal boundary $Y[0]$ will be  
\begin{equation}\label{Gh}
G^{Y[0]}_{\mathbf 0, \mathbf x} = \max_{i \in \mathbb{Z}} h_i + G_{(i,1), \mathbf x}.
\end{equation}
We note that because of the horizontal boundary, \eqref{Gh} may not be well defined. For example, when $Y[0] \equiv 0$, the last-passage value $G^{Y[0]}_{0,x}$ may be infinity. We do not need to address this here, because, with the random weights which we introduce later, the last-passage value and the geodesic are finite almost surely.

For $\mathbf z\in \mathbb{Z}\times \mathbb{Z}_{\geq 0}$,  define the increments $I$, $J$ and the dual weights $\wc Y$ as 
\begin{align}
I_{[\![\mathbf z, \mathbf z+\mathbf e_1]\!]}  &= \begin{cases}
Y_{\mathbf z+\mathbf e_1}  \qquad & \text{ if }\mathbf z\cdot \mathbf e_2 =0\\
G^{{Y[0]}} _{\mathbf 0,\mathbf z+\mathbf e_1} - G^{{Y[0]}}_{\mathbf 0,\mathbf z}  \qquad & \text{ if } \mathbf z\cdot \mathbf e_2 \geq 1
\end{cases}\nonumber\\ 
J_{[\![\mathbf z, \mathbf z+\mathbf e_2]\!]} &= G^{{Y[0]}} _{\mathbf 0, \mathbf z+\mathbf e_2} - G^{Y[0]} _{\mathbf 0,\mathbf z} \label{def_inc}\\ 
\wc{Y}_{\mathbf z} &= I_{[\![\mathbf z, \mathbf z+\mathbf e_1]\!]} \wedge J_{[\![\mathbf z, \mathbf z+\mathbf e_2]\!]}\nonumber
\end{align}
Note that these quantities depend on the starting point of the last-passage value, which in the definition above, is the origin. However, we have chosen to omit this information from our notation to minimize the clustering of symbols.

Furthermore, we will rewrite our weights by their $ y$-coordinates. Let $Y[m] = (Y_{(j,m)})_{j\in\mathbb{Z}}$ for $m \in \mathbb{Z}$, and for $m\in \mathbb{Z}_{\geq 0}$, let 
\begin{align}
I[m] &= (I_{[\![(j-1, m), (j, m)]\!]})_{j\in \mathbb{Z}}\label{def_I}\\
\wc Y[m] &= (\wc Y_{(j-1,m)})_{j\in \mathbb{Z}}.\label{shift}
\end{align}
Notice the shift in the $x$-coordinate in \eqref{shift}.

Recall the queueing notation from \cite{Fan-Sep-20}, the inter-departure time (denoted by $D$) and the sojourn time (denoted by $S$) of a queue  with inter-arrival $Y[0]$ and service time $Y[1]$ are defined by 
\begin{align}
D^{(1)}(Y[0], Y[1]) &= \big(G^{Y[0]} _{\mathbf 0, (j,1)} - G^{Y[0]} _{\mathbf 0, (j-1,1)}\big)_{j\in \mathbb{Z}} =  I[1].\label{D2}\\
S(Y[0], Y[1]) &= \big(G^{Y[0]} _{\mathbf 0, (j,1)} - G^{Y[0]} _{\mathbf 0, (j,0)}\big)_{j\in \mathbb{Z}}.\nonumber
\end{align}
And in general, the inter-departure time with $n$ service stations, $D^{(n)}(Y[0], \mydots, Y[n])$, is defined by iteratively applying \eqref{D2}  $n$ times , i.e.
\begin{equation}\label{Dn}
D^{(n)}(Y[0], \mydots, Y[n])= D^{(1)}\Big(\mydots D^{(1)}\Big(D^{(1)}\Big(Y[0], Y[1]\Big),Y[2]\Big) \mydots , Y[n]\Big).
\end{equation}

The first proposition below states that the inter-departure time of a queue with $n$ service stations, $D^{(n)}(Y[0], \mydots, Y[n])$, can be seen as the last-passage increments on level $n$ under the environment $\{Y[m]\}_{m=0}^n$.
\begin{proposition} \label{D_to_LPP} Fix $n \geq 1$, and suppose $\{Y[m]\}_{m=0}^n$ are chosen so that the following queueing operations and last-passage values are well defined, we have
$$D^{(n)}(Y[0], \mydots, Y[n]) = \big(G^{Y[0]} _{\mathbf 0, (j,n)} - G^{Y[0]} _{\mathbf 0, (j-1,n)}\big)_{j\in \mathbb{Z}} = I[n].$$
\end{proposition}

Additionally, the sojourn time between service stations $n-1$ and $n$ is equal to the vertical increments between levels $n-1$ and $n$.
\begin{proposition} \label{S_to_LPP} Fix $n \geq 1$, and suppose $\{Y[m]\}_{m=0}^n$ are chosen so that the following queueing operations and last-passage values are well defined, we have
$$S(D^{(n-1)}(Y[0], \mydots, Y[n-1]), Y[n]) = \big(G^{Y[0]} _{\mathbf 0, (j,n)} - G^{Y[0]} _{\mathbf 0, (j,n-1)}\big)_{j\in \mathbb{Z}}.$$
\end{proposition}

The two aforementioned propositions are direct consequences of a fundamental result concerning nested LPP, which we shall now define and introduce.  We define a nested LPP system as $G^{I[n-1]}_{(0, n-1), \bbullet}$ which starts at $(0, n-1)$ and utilizes the horizontal boundary $I[n-1]$. The term ``nested LPP" is because of the fact that its boundary values are computed from the (outer) LPP $G^{Y[0]}_{\mathbf 0, \bbullet}$. 

The next proposition shows that the increments of $G^{I[n-1]} _{(0, n-1), \bbullet}$ and $G^{Y[0]} _{\mathbf 0, \bbullet}$ are equal between level $n$ and $n-1$.
It is worth noting that the proof is more complicated than the analogous results for the southwest and anti-diagonal boundaries, due to the presence of a horizontal boundary. 
\begin{proposition}\label{A_nest}
For each $n \in \mathbb{Z}_{>0}$ and $j\in \mathbb{Z}$, suppose that both $G^{Y[0]}_{\mathbf 0, (j, n)}$ and $G^{I[n-1]} _{(0, n-1), (j, n)}$ are finite with finite geodesics. Then, it holds that 
\begin{align*}
G^{Y[0]}_{\mathbf 0, (j, n)} - G^{Y[0]}_{\mathbf 0,(j-1, n)} &= G^{I[n-1]} _{(0, n-1), (j, n)} - G^{I[n-1]} _{(0, n-1), (j-1, n)} \\
G^{Y[0]}_{\mathbf 0, (j, n)} - G^{Y[0]}_{\mathbf 0,(j, n-1)} &= G^{I[n-1]} _{(0, n-1), (j, n)} - G^{I[n-1]} _{(0, n-1), (j, n-1)}.
\end{align*}
\end{proposition}

\begin{proof}
The proof follows from Lemma A.4 in \cite{Fan-Sep-20}. For the setup, our $I[n-1]$ and $\{Y_{(j,n)}\}_{j\in \mathbb{Z}^2}$ play the roles of $I$ and $\omega$ from \cite[Lemma A.4]{Fan-Sep-20}. The horizontal increments $\{G^{Y[0]}_{\mathbf 0, (j, n)} - G^{Y[0]}_{\mathbf 0,(j-1, n)}\}_{j\in \mathbb{Z}}$ and vertical increments $\{G^{Y[0]}_{\mathbf 0, (j, n)} - G^{Y[0]}_{\mathbf 0,(j, n-1)}\}_{j\in \mathbb{Z}}$  play the roles of $\wt{I}$ and $J$ of \cite[Lemma A.4]{Fan-Sep-20}. For the rest of the proof, we will be adapting this notation with $I, \omega, \wt{I}$ and $J$ from \cite[Lemma A.4]{Fan-Sep-20}.

We verify the three assumptions in \cite[Lemma A.4]{Fan-Sep-20}. Assumption \cite[(A.14)]{Fan-Sep-20} is used to guarantee that the last-passage values are well defined, which we have assumed as part of the statement of our proposition. Assumption \cite[(A.16)]{Fan-Sep-20} requires $J_j = \omega_j$ for infinitely many $j$'s. The equality holds whenever  the geodesic between $(0,0)$ and $(j, n)$ goes through $(j, n-1)$. This happens from the fact that our geodesics have finite length by assumption, and any sub-path of a geodesic is also a geodesic. Lastly, assumption \cite[(A.15)]{Fan-Sep-20} states that  
$$\wt{I}_j = \omega_j + (I_j - J_{j-1})^+ \qquad J_j = \omega_j + (I_j - J_{j-1})^- \qquad \text{ for each }j\in \mathbb{Z}.$$
This is well-known for the increments of the last-passage value. To see this, by definition we have 
\begin{equation}\label{CGM_id}
G^{Y[0]}_{\mathbf 0, \mathbf x} = \omega_\mathbf x + G^{Y[0]}_{\mathbf 0, \mathbf x-\mathbf e_1} \vee G^{Y[0]}_{\mathbf 0, \mathbf x-\mathbf e_2},
\end{equation}
subtracting both sides by $G^{Y[0]}_{\mathbf 0, \mathbf x-\mathbf e_1}$ we obtain $\wt{I}_j = \omega_j + (I_j - J_{j-1})^+$. The other statement follows from subtracting both sides of \eqref{CGM_id} by $G^{Y[0]}_{\mathbf 0, \mathbf x-\mathbf e_2}$.

Once these are verified, Lemma A.4 of \cite{Fan-Sep-20} states that the increments $\wt I$ and $J$, obtained from $G^{Y[0]}_{\mathbf 0, \bbullet}$ must be equal to the increments from $G^{I[n-1]} _{(0, n-1), \bbullet}$, which are the inter-departure and the sojourn times for a queue by definition. 
\end{proof}
 
Now we are ready to prove Proposition \ref{D_to_LPP} and Proposition \ref{S_to_LPP}.
\begin{proof}[Proof of Proposition \ref{D_to_LPP}]
We argue by induction. The base case when $n=1$ holds by definition. Then, suppose the statement holds for $n-1$, i.e.
$$D^{(n-1)}(Y[0], \mydots, Y[n-1]) = I[n-1].$$
From the induction hypothesis and definition \eqref{Dn}, 
$$D^{(n)}(Y[0], \mydots, Y[n]) = \big(G^{I[n-1]} _{(0,n-1), (j,n)} - G^{I[n-1]}_{(0,n-1), (j-1,n)}\big)_{j\in \mathbb{Z}}.$$
Finally, Proposition \ref{A_nest} states that the right hand side above is equal to $I[n]$ 
. And we have completed the proof.
\end{proof}

\begin{proof}[Proof of Proposition \ref{S_to_LPP}]
Using Proposition \ref{D_to_LPP}, we have
$$S(D^{(n-1)}(Y[0], \mydots, Y[n-1]), Y[n]) = \big(G^{I[n-1]} _{(0,n-1), (j,n)} - G^{I[n-1]} _{(0,n-1), (j,n-1)}\big)_{j\in \mathbb{Z}}$$
Then, applying Proposition \ref{A_nest} to the right side above gives us the desired result. 
\end{proof}

Because of Proposition \ref{D_to_LPP} and Proposition \ref{S_to_LPP}, we can now avoid explicitly distinguishing between the notations used for queues and LPP from a horizontal boundary.

Next, we state a queueing identity. Fix $k\geq 1$, and let $G^{Y[-k]} _{(0, -k), \bbullet}$ be the last-passage value starting from $(0, -k)$ with boundary $Y[-k]$.
The proposition below states that for any $k\geq 1$, if we replace the environment $Y[0]$ and $Y[1]$ by $\wc Y[0]$ and $I[1]$, then the last-passage increments of $G^{Y[-k]} _{(0, -k), \bbullet}$ at level $n \geq 1$ does not change. This result directly follows from Lemma 4.4 of \cite{Fan-Sep-20}.
\begin{proposition}\label{2level}
Fix $k, n \geq 1$, it holds that
$$D^{(n+k)}(Y[-k], \mydots, Y[0], Y[1], \mydots, Y[n]) = D^{(n+k)}(Y[-k], \mydots, \wc Y[0], I[1], \mydots, Y[n]).$$
\end{proposition}
\begin{proof}
Lemma 4.5 from \cite{Fan-Sep-20} states that for arbitrary weights $A[-1], A[0], A[1]$ such that as long as the queueing operations 
are well-defined, it holds that
\begin{equation}\label{fsid}
D^{(1)}(D^{(1)}(A[-1], A[0]), A[1]) = D^{(1)}(D^{(1)}(A[-1], \wc A [0]), D^{(1)}(A[0], A[1])).
\end{equation}
Then, in the calculation below, the second equality is obtained by letting
$$A[-1] = D^{(k-1)}(Y[-k], \mydots, Y[-1]), \qquad A[0] = Y[0], \qquad A[1] = Y[1]$$ and applying the identity \eqref{fsid}. The third equality holds since $D^{(1)}(Y[0], Y[1]) = I[1]$ by definition. We have 
\begin{align*}
&D^{(n+k)}(Y[-k], \mydots, Y[0], Y[1], \mydots, Y[n])\\
&= D^{(1)}(D^{(1)}(\mydots D^{(1)}(D^{(1)}(D^{(k-1)}(Y[-k], \mydots Y[-1]), Y[0]), Y[1]), \mydots), Y[n])\\
&= D^{(1)}(D^{(1)}(\mydots D^{(1)}(D^{(1)}(D^{(k-1)}(Y[-k], \mydots, Y[-1]), \wc Y[0]), D^{(1)}(Y[0], Y[1])), \mydots), Y[n])\\
&= D^{(1)}(D^{(1)}(\mydots D^{(1)}(D^{(1)}(D^{(k-1)}(Y[-k], \mydots, Y[-1]), \wc Y[0]), I[1]), \mydots), Y[n])\\
& =  D^{(n+k)}(Y[-k], \mydots, \wc Y[0], I[1], \mydots, Y[n]),
\end{align*}
and we have finished the proof.
\end{proof}

In the final part of this section, we introduce random variables into the discussion.
Let us define an infinite sequence $\{s_m\}_{m\in \mathbb{Z}_{\geq 0}}$ with $0< s_0< s_1 \leq  s_2\leq \mydots  \leq 1$.
For $m \in \mathbb{Z}_{\geq 0}$, let the entries of $Y[m]$ be a sequence of i.i.d~$\Exp(s_m)$ random variables. Since $s_0 < s_m$ for every $m \geq 1$, according to Lemma A.1 of \cite{Fan-Sep-20}, the last-passage value $G^{Y[0]}_{\mathbf 0,\bbullet}$ and its unique geodesic are almost surely finite. Thus, the increments and the dual weights from \eqref{def_inc} are well defined.  
Next, we state the following proposition about the independent properties for the increments and the dual weights in \eqref{def_inc}. 

\begin{figure}[t]

\begin{center}

\tikzset{every picture/.style={line width=0.75pt}} 

\begin{tikzpicture}[x=0.75pt,y=0.75pt,yscale=-1,xscale=1]

\draw  [color={rgb, 255:red, 74; green, 144; blue, 226 }  ,draw opacity=1 ][line width=1.5]  (76.5,102.9) -- (183,102.9) -- (183,129.6) ;
\draw   (183,129.6) -- (298.8,129.6) -- (298.8,181.9) ;
\draw   (298.8,181.9) -- (349.93,181.9) -- (349.93,200.75) ;
\draw  [color={rgb, 255:red, 155; green, 155; blue, 155 }  ,draw opacity=1 ][dash pattern={on 2.53pt off 3.02pt}][line width=2.25]  (77,129.6) -- (298.8,129.6) -- (298.8,181.9) ;
\draw  [color={rgb, 255:red, 155; green, 155; blue, 155 }  ,draw opacity=1 ][dash pattern={on 2.53pt off 3.02pt}][line width=2.25]  (298.8,181.9) -- (349.93,181.9) -- (349.93,200.25) ;
\draw  [fill={rgb, 255:red, 0; green, 0; blue, 0 }  ,fill opacity=1 ] (180.58,102.9) .. controls (180.58,101.56) and (181.66,100.48) .. (183,100.48) .. controls (184.34,100.48) and (185.42,101.56) .. (185.42,102.9) .. controls (185.42,104.24) and (184.34,105.33) .. (183,105.33) .. controls (181.66,105.33) and (180.58,104.24) .. (180.58,102.9) -- cycle ;
\draw    (349.93,200.25) -- (429.5,200.01) ;
\draw [shift={(431.5,200.01)}, rotate = 179.83] [color={rgb, 255:red, 0; green, 0; blue, 0 }  ][line width=0.75]    (10.93,-3.29) .. controls (6.95,-1.4) and (3.31,-0.3) .. (0,0) .. controls (3.31,0.3) and (6.95,1.4) .. (10.93,3.29)   ;
\draw [color={rgb, 255:red, 155; green, 155; blue, 155 }  ,draw opacity=1 ][line width=2.25]  [dash pattern={on 2.53pt off 3.02pt}]  (349.93,200.25) -- (408.5,200.04) ;
\draw [shift={(412.5,200.03)}, rotate = 179.79] [color={rgb, 255:red, 155; green, 155; blue, 155 }  ,draw opacity=1 ][line width=2.25]    (17.49,-5.26) .. controls (11.12,-2.23) and (5.29,-0.48) .. (0,0) .. controls (5.29,0.48) and (11.12,2.23) .. (17.49,5.26)   ;

\draw (49.5,136.9) node [anchor=north west][inner sep=0.75pt]    {$\mathcal Y^{K-1}$};
\draw (53,87.4) node [anchor=north west][inner sep=0.75pt]    {$\mathcal Y^{K}$};
\draw (185,78.4) node [anchor=north west][inner sep=0.75pt]    {$(a^{*} ,K)$};

\end{tikzpicture}

\captionsetup{width=.8\linewidth}
\caption{An illustration for the induction 
in Proposition \ref{BK}. The value $a^*$ corresponds to the $x$-coordinate for which the path $\mathcal{Y}^K$ takes its first $-\mathbf  e_2$ step. As defined in \eqref{dr1}, the LPP increments along the blue segment and the dual weights below are measurable relative to the random environment situated to the right of the vertical line $x=a^*$.}\label{fig:bk}
\end{center}

\end{figure}
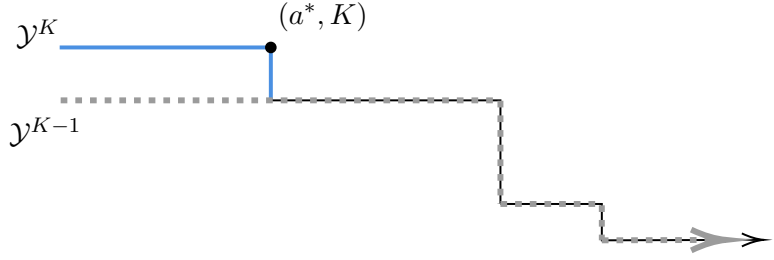
\begin{proposition}\label{BK}
Let $\mathcal{Y} = \{\mathbf y_i\}_{i\in \mathbb{Z}}$ be any down-right paths in $\mathbb{Z}\times \mathbb{Z}_{\geq 0}$ with $\mathbf y_{i+1}-\mathbf y_i\in \{\mathbf e_1, -\mathbf e_2\}$. Then, calculated from $G^{Y[0]}_{\mathbf 0,\bbullet}$, the collection of increments $I$, $J$ along $\mathcal{Y}$
and the dual weights $\wc Y$ below $\mathcal{Y}$
are mutually independent. 
\end{proposition}
\begin{proof}
This follows the standard ``corner flipping'' induction for the KPZ lattice models. Lemma B.2 of \cite{Fan-Sep-20} verified the base case when the down-right path $\mathcal{Y} = \mathcal{Y}^1$ is contained between the horizontal lines $y = 0$ and $y=1$. 

Next, we proceed with the induction step which is illustrated in Figure \ref{fig:bk}. Fix $K$, suppose we have a down-right path  $\mathcal{Y}^K = \{\mathbf y^K_i\}_{i \in \mathbb{Z}}$ that starts at $(-\infty, K)$. The path $\mathcal{Y}^K$ will take a ``$-\mathbf e_2$" step at points $(a^*, K)$  for some $a^* \in \mathbb{Z}\cup\{\infty\}$ (the case $a^* = \infty$ means $\mathcal{Y}^K$ is a horizontal line). 
The inductive hypothesis states that for each down-right path inside $\mathbb{Z}\times [0,K-1]$, the independence result holds. In particular, we will be using the down-right path $\mathcal{Y}^{K-1}$ which goes from $(-\infty, K-1)$ to $(a^*, K-1)$ then follows $\mathcal{Y}^K$. 

Let us define a nested LPP $G^{I[K-1]}_{(0, K-1), \bbullet}$ that has a horizontal boundary on the line $y = {K-1}$ with the boundary value is given by $I[K-1]$. First, we assume that $a^* < \infty$. 
Applying Lemma B.2 of \cite{Fan-Sep-20} to $G^{I[K-1]}_{(0, K-1), \bbullet}$ along the down-right path 
$(-\infty, K) \to (a^*, K) \to (a^*, K-1) \to (\infty, K-1)$, and subsequently making use of Proposition \ref{A_nest}, we see that 
\begin{equation} \label{dr1}
\Big\{G^{Y[0]}_{\mathbf 0, (j, K)} - G^{Y[0]}_{\mathbf 0, (j-1, K)}\Big\}_{j \leq a^*} \bigcup \Big\{\wc Y_{(j, K-1)}\Big\}_{j< a^*} \bigcup \Big\{G^{Y[0]}_{\mathbf 0, (a^*, K)} - G^{Y[0]}_{\mathbf 0, (a^*, K-1)} \Big\}\end{equation}
are mutually independent. In addition, since these random variables are measurable with respect to $\sigma(\{I[K-1]_j, Y_{(j, K)}\}_{j\leq a^*})$, together with the inductive hypothesis applied to $\mathcal{Y}^{K-1}$, we obtain the claimed independence in our proposition. 
Lastly, when considering the case $a^* = \infty$, the same argument applies, albeit with the exception that the vertical increment, which is the third term in \eqref{dr1}, is omitted.
\end{proof}

\section{Proof of Theorem \ref{main}}\label{pf_main1}

Again, we will present our arguments for the CGM, and the modification for the other models will be addressed in Section \ref{ig_poly} and Section \ref{more_models}.

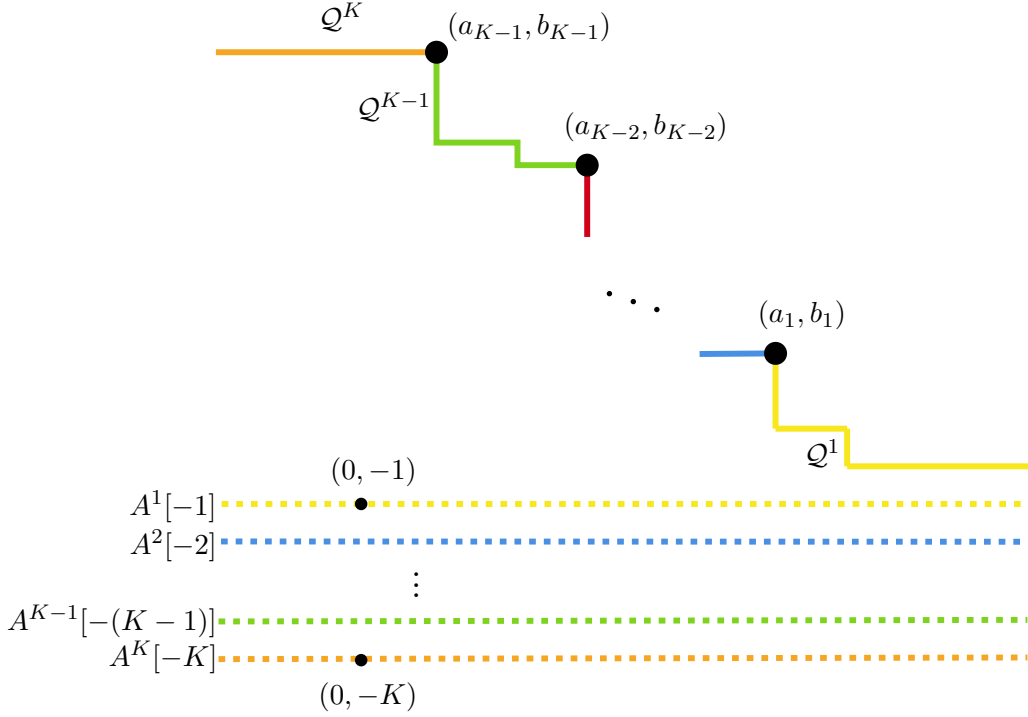
\begin{figure}[t]
\begin{center}

\tikzset{every picture/.style={line width=0.75pt}} 

\begin{tikzpicture}[x=0.75pt,y=0.75pt,yscale=-0.95,xscale=0.95]

\draw [color={rgb, 255:red, 245; green, 166; blue, 35 }  ,draw opacity=1 ][line width=2.25]    (119.9,40) -- (237,40) ;

\draw [color={rgb, 255:red, 126; green, 211; blue, 33 }  ,draw opacity=1 ][line width=2.25]    (237,40) -- (237,88) -- (280, 88) -- (280,100) --  (317,100);
\draw [color={rgb, 255:red, 208; green, 2; blue, 27 }  ,draw opacity=1 ][line width=2.25]    (317,100) -- (317,138.14) ;
\draw [color={rgb, 255:red, 248; green, 231; blue, 28 }  ,draw opacity=1 ][line width=2.25]    (417,200) -- (417,240) ;
\draw [color={rgb, 255:red, 248; green, 231; blue, 28 }  ,draw opacity=1 ][line width=2.25]    (417,240) -- (455,240) ;
\draw [color={rgb, 255:red, 248; green, 231; blue, 28 }  ,draw opacity=1 ][line width=2.25]    (455,240) -- (455,260) ;
\draw [color={rgb, 255:red, 248; green, 231; blue, 28 }  ,draw opacity=1 ][line width=2.25]    (455,260) -- (553.1,260) ;
\draw [color={rgb, 255:red, 74; green, 144; blue, 226 }  ,draw opacity=1 ][line width=2.25]    (376.67,200.33) -- (417,200) ;
\draw [color={rgb, 255:red, 248; green, 231; blue, 28 }  ,draw opacity=1 ][line width=2.25]  [dash pattern={on 2.53pt off 3.02pt}]  (122.94,280) -- (553.1,280) ;
\draw [color={rgb, 255:red, 74; green, 144; blue, 226 }  ,draw opacity=1 ][line width=2.25]  [dash pattern={on 2.53pt off 3.02pt}]  (122.61,299.92) -- (551.39,300.08) ;
\draw [color={rgb, 255:red, 126; green, 211; blue, 33 }  ,draw opacity=1 ][line width=2.25]  [dash pattern={on 2.53pt off 3.02pt}]  (121.28,342.43) -- (550.72,341.57) ;
\draw [color={rgb, 255:red, 245; green, 166; blue, 35 }  ,draw opacity=1 ][line width=2.25]  [dash pattern={on 2.53pt off 3.02pt}]  (121.28,362.43) -- (550.72,361.57) ;
\draw  [fill={rgb, 255:red, 0; green, 0; blue, 0 }  ,fill opacity=1 ] (194.21,280) .. controls (194.21,278.46) and (195.46,277.21) .. (197,277.21) .. controls (198.54,277.21) and (199.79,278.46) .. (199.79,280) .. controls (199.79,281.54) and (198.54,282.79) .. (197,282.79) .. controls (195.46,282.79) and (194.21,281.54) .. (194.21,280) -- cycle ;
\draw  [fill={rgb, 255:red, 0; green, 0; blue, 0 }  ,fill opacity=1 ] (311.42,100) .. controls (311.42,96.92) and (313.92,94.42) .. (317,94.42) .. controls (320.08,94.42) and (322.58,96.92) .. (322.58,100) .. controls (322.58,103.08) and (320.08,105.58) .. (317,105.58) .. controls (313.92,105.58) and (311.42,103.08) .. (311.42,100) -- cycle ;
\draw  [fill={rgb, 255:red, 0; green, 0; blue, 0 }  ,fill opacity=1 ] (231.42,40) .. controls (231.42,36.92) and (233.92,34.42) .. (237,34.42) .. controls (240.08,34.42) and (242.58,36.92) .. (242.58,40) .. controls (242.58,43.08) and (240.08,45.58) .. (237,45.58) .. controls (233.92,45.58) and (231.42,43.08) .. (231.42,40) -- cycle ;
\draw  [fill={rgb, 255:red, 0; green, 0; blue, 0 }  ,fill opacity=1 ] (411.42,200) .. controls (411.42,196.92) and (413.92,194.42) .. (417,194.42) .. controls (420.08,194.42) and (422.58,196.92) .. (422.58,200) .. controls (422.58,203.08) and (420.08,205.58) .. (417,205.58) .. controls (413.92,205.58) and (411.42,203.08) .. (411.42,200) -- cycle ;
\draw  [fill={rgb, 255:red, 0; green, 0; blue, 0 }  ,fill opacity=1 ] (194.21,363) .. controls (194.21,361.46) and (195.46,360.21) .. (197,360.21) .. controls (198.54,360.21) and (199.79,361.46) .. (199.79,363) .. controls (199.79,364.54) and (198.54,365.79) .. (197,365.79) .. controls (195.46,365.79) and (194.21,364.54) .. (194.21,363) -- cycle ;

\draw (172.93,12.2) node [anchor=north west][inner sep=0.75pt]    {$\mathcal{Q}^K$};
\draw (193,60) node [anchor=north west][inner sep=0.75pt]    {$\mathcal{Q}^{K-1}$};
\draw (430.73,244.2) node [anchor=north west][inner sep=0.75pt]    {$\mathcal{Q}^1$};
\draw (179,253.4) node [anchor=north west][inner sep=0.75pt]    {$(0,-1)$};
\draw (72.5,271.4) node [anchor=north west][inner sep=0.75pt]    {$A^{1}[ -1]$};
\draw (72.67,292.07) node [anchor=north west][inner sep=0.75pt]    {$A^{2}[ -2]$};
\draw (61.5,351.07) node [anchor=north west][inner sep=0.75pt]    {$A^{K}[ -K]$};
\draw (241,17.4) node [anchor=north west][inner sep=0.75pt]    {$( a_{K-1} ,b_{K-1})$};
\draw (303,70.4) node [anchor=north west][inner sep=0.75pt]    {$( a_{K-2} ,b_{K-2})$};
\draw (406,170.4) node [anchor=north west][inner sep=0.75pt]    {$( a_{1} ,b_{1})$};
\draw (172.8,373) node [anchor=north west][inner sep=0.75pt]    {$( 0,-K)$};
\draw (222,306.4) node [anchor=north west][inner sep=0.75pt]  [font=\Large]  {$\vdots $};
\draw (321,155.4) node [anchor=north west][inner sep=0.75pt]  [font=\Huge]  {$\ddots$};
\draw (7.5,332.07) node [anchor=north west][inner sep=0.75pt]    {$A^{K-1}[-(K-1)]$};

\end{tikzpicture}

\captionsetup{width=.8\linewidth}
\caption{Setup for the last-passage values $\big\{G^{A^k[-k]}_{(0,-k), \bullet}\big\}_{k=1}^K$ and the collection of increments $\{\mathcal{Q}^k\}_{k=1}^K$. The collection of horizontal boundaries $\{A^k[-k]\}_{k=1}^K$ are shown as dotted lines on the bottom.}\label{cgm_setup}
\end{center}
\end{figure}
Firstly, we will  demonstrate an independence result for the increments from a LPP that starts with a horizontal boundary. We now describe the setup, and this is illustrated in Figure \ref{cgm_setup}.
Fix $K\in \mathbb{Z}_{>0}$, and $0<\rho_1<\rho_2 <  \mydots < \rho_K<1$. For $k =  1, {\mydots}, K$, let $\{A^k[-k]_j\}_{j\in \mathbb{Z}}$ be a collection of i.i.d.~$\Exp(1-\rho_k)$ random variables, where the weight $A^k[-k]_j$ is attached with the vertex $(j, -k)$. Define $A^k[-k] = (A^k[-k]_j)_{j\in \mathbb{Z}}$, and we assume $A^1[-1], {\mydots}, A^K[-K]$ are all independent. Above the $x$-axis, let $\{\omega_\mathbf z\}_{\mathbf z\in \mathbb{Z}\times \mathbb{Z}_{\geq 0}}$ be a collection of i.i.d~$\Exp(1)$ random variables attached to the vertices. Furthermore, for $n\in \mathbb{Z}_{\geq 0}$, let us define $\omega[n] = (\omega_{(j, n)})_{j\in \mathbb{Z}}$ which are the weights on the horizontal line $y=n$. 
For $k = 1, {\mydots}, K$, let  $G^{A^k[-k]} _{(0, -k), \bbullet}$ to be the last-passage value with the horizontal boundary $A^k[-k]$.

Referring to Figure \ref{cgm_setup}, 
fix a down-right path $\mathcal{Y} = \{\mathbf y_i\}_{i\in \mathbb{Z}}$ with $\mathbf y_{i+1} - \mathbf y_i \in \{\mathbf e_1, -\mathbf e_2\}$  inside $\mathbb{Z} \times \mathbb{Z}_{\geq 0}$. 
Fix a sequence $i_1> i_2 > \mydots > i_{K-1}$, define the collections of increments
\begin{align}
&\mathcal{Q}^1 = \Big\{G^{A^1[-1]} _{(0,-1),\mathbf y_{i+1}} -  G^{A^1[-1]} _{(0,-1),\mathbf y_{i}} :  i_1 \leq  i < \infty \Big\} 
\nonumber\\
&\mathcal{Q}^k = \Big\{G^{A^k[-k]} _{(0,-k),\mathbf y_{i+1}} - G^{A^k[-k]} _{(0,-k),\mathbf y_{i}} : i_{k}\leq  i < i_{k-1} \Big\} \quad \text{ for }k \in \{2, \mydots, K-1\}\label{def_Qk}\\
&\mathcal{Q}^K = \Big\{G^{A^K[-K]} _{(0,-K),\mathbf y_{i+1}} -  G^{A^K[-K]} _{(0,-K),\mathbf y_{i}}: -\infty<  i  < i_{K-1} \Big\}. \nonumber
\end{align}

Then, we have the following independence result. 
\begin{theorem}\label{ind_B}
The collection of random variables $\mathcal{Q}^1 \cup \mathcal{Q}^2 \cup {\mydots} \cup \mathcal{Q}^K$ are mutually independent. 
\end{theorem}

\begin{proof}
To simplify the notation, let $\mathbf y_{i_k} = (a_k, b_k)$ for $k = 1, \mydots, K-1$.
 And for the rest of the proof, we will adapt the notation from Section \ref{q_basic}, for $k \in \{1, {\mydots}, K\}$, $n \in \mathbb{Z}_{\geq 0}$ and $j \in \mathbb{Z}$, denote
$$D^{(n+k)}(A^k[-k], {\mydots}, A^1[-1], \omega, {\mydots}, \omega[n]) = \big(G^{A^k[-k]} _{(0, -k), (j, n)}- G^{A^k[-k]} _{(0, -k), (j-1, n)}\big)_{j \in \mathbb{Z}}.$$

First, we will show that random variables inside $\mathcal{Q}^1$ are independent, furthermore, they are independent with $\sigma(\mathcal{Q}^2, \mathcal{Q}^3, {\mydots}, \mathcal{Q}^K)$. 
To do this, note that $\mathcal{Q}^2, {\mydots}, \mathcal{Q}^K$ can be viewed as functions of
\begin{equation}\label{step1_before}
\Big\{D^{(n+k)}(A^k[-k], {\mydots}, A^1[-1], \omega[0], \omega[1], {\mydots}, \omega[n])_j, \omega_{(j, n+1)}\Big\}_{j \leq a_1, k \in \{2, {\mydots} K\}, n \geq b_1}.
\end{equation}
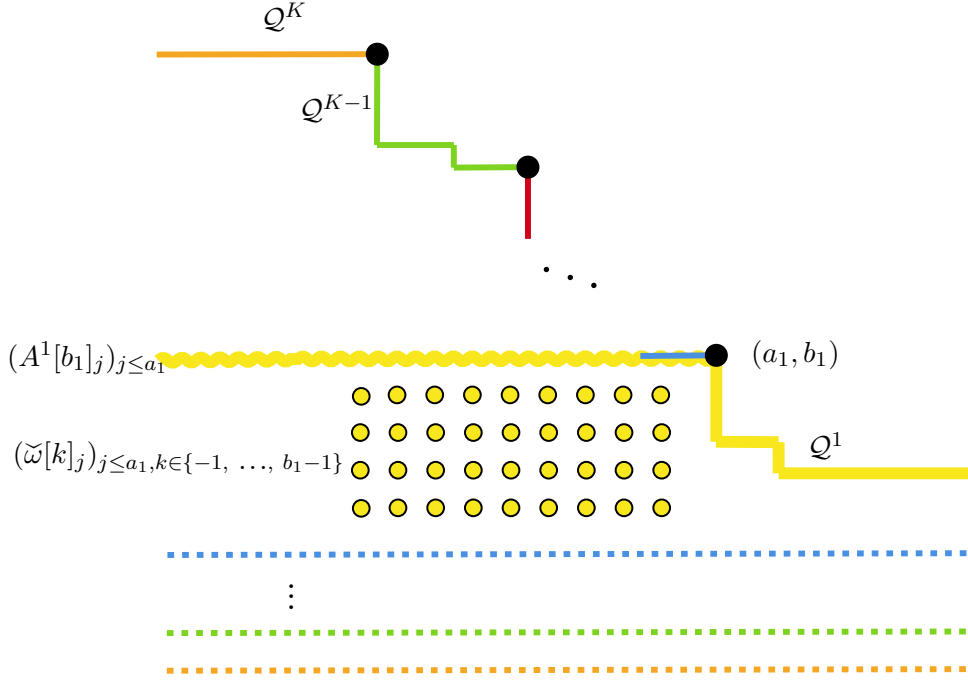
\begin{figure}
\begin{center}

\tikzset{every picture/.style={line width=0.75pt}} 

\begin{tikzpicture}[x=0.75pt,y=0.75pt,yscale=-0.95,xscale=0.95]

\draw [color={rgb, 255:red, 248; green, 231; blue, 28 }  ,draw opacity=1 ][line width=4.5]    (430,234) -- (430,278.24) ;
\draw [color={rgb, 255:red, 248; green, 231; blue, 28 }  ,draw opacity=1 ][line width=4.5]    (463.07,295) -- (566.1,295) ;
\draw [color={rgb, 255:red, 248; green, 231; blue, 28 }  ,draw opacity=1 ][line width=4.5]    (133.89,235.02) .. controls (135.54,233.33) and (137.21,233.32) .. (138.89,234.97) .. controls (140.57,236.62) and (142.24,236.6) .. (143.89,234.92) .. controls (145.54,233.23) and (147.2,233.21) .. (148.89,234.86) .. controls (150.57,236.51) and (152.24,236.49) .. (153.89,234.81) .. controls (155.54,233.13) and (157.21,233.11) .. (158.89,234.76) .. controls (160.57,236.41) and (162.24,236.39) .. (163.89,234.71) .. controls (165.54,233.03) and (167.21,233.01) .. (168.89,234.66) .. controls (170.57,236.31) and (172.24,236.29) .. (173.89,234.61) .. controls (175.54,232.93) and (177.21,232.91) .. (178.89,234.56) .. controls (180.57,236.21) and (182.24,236.19) .. (183.89,234.51) .. controls (185.54,232.83) and (187.21,232.81) .. (188.89,234.46) .. controls (190.57,236.11) and (192.24,236.09) .. (193.89,234.41) .. controls (195.54,232.73) and (197.21,232.71) .. (198.89,234.36) .. controls (200.57,236.01) and (202.24,235.99) .. (203.89,234.31) -- (207.95,234.27) -- (207.95,234.27) .. controls (209.62,232.6) and (211.28,232.59) .. (212.95,234.26) .. controls (214.62,235.92) and (216.28,235.92) .. (217.95,234.25) .. controls (219.62,232.58) and (221.28,232.58) .. (222.95,234.25) .. controls (224.62,235.91) and (226.28,235.91) .. (227.95,234.24) .. controls (229.62,232.57) and (231.28,232.57) .. (232.95,234.24) .. controls (234.62,235.9) and (236.28,235.9) .. (237.95,234.23) .. controls (239.62,232.56) and (241.28,232.56) .. (242.95,234.22) .. controls (244.62,235.89) and (246.28,235.89) .. (247.95,234.22) .. controls (249.62,232.55) and (251.28,232.55) .. (252.95,234.21) .. controls (254.62,235.88) and (256.28,235.88) .. (257.95,234.21) .. controls (259.62,232.54) and (261.28,232.54) .. (262.95,234.2) .. controls (264.62,235.86) and (266.28,235.86) .. (267.95,234.19) .. controls (269.62,232.52) and (271.28,232.52) .. (272.95,234.19) .. controls (274.62,235.85) and (276.28,235.85) .. (277.95,234.18) .. controls (279.62,232.51) and (281.28,232.51) .. (282.95,234.18) .. controls (284.62,235.84) and (286.28,235.84) .. (287.95,234.17) .. controls (289.62,232.5) and (291.28,232.5) .. (292.95,234.16) .. controls (294.62,235.83) and (296.28,235.83) .. (297.95,234.16) .. controls (299.62,232.49) and (301.28,232.49) .. (302.95,234.15) .. controls (304.62,235.82) and (306.28,235.82) .. (307.95,234.15) .. controls (309.62,232.48) and (311.28,232.48) .. (312.95,234.14) .. controls (314.62,235.8) and (316.28,235.8) .. (317.95,234.13) .. controls (319.62,232.46) and (321.28,232.46) .. (322.95,234.13) .. controls (324.62,235.79) and (326.28,235.79) .. (327.95,234.12) .. controls (329.62,232.45) and (331.28,232.45) .. (332.95,234.12) .. controls (334.62,235.78) and (336.28,235.78) .. (337.95,234.11) .. controls (339.62,232.44) and (341.28,232.44) .. (342.95,234.1) .. controls (344.62,235.77) and (346.28,235.77) .. (347.95,234.1) .. controls (349.62,232.43) and (351.28,232.43) .. (352.95,234.09) .. controls (354.62,235.76) and (356.28,235.76) .. (357.95,234.09) .. controls (359.62,232.42) and (361.28,232.42) .. (362.95,234.08) .. controls (364.62,235.74) and (366.28,235.74) .. (367.95,234.07) .. controls (369.62,232.4) and (371.28,232.4) .. (372.95,234.07) .. controls (374.62,235.73) and (376.28,235.73) .. (377.95,234.06) .. controls (379.62,232.39) and (381.28,232.39) .. (382.95,234.06) .. controls (384.62,235.72) and (386.28,235.72) .. (387.95,234.05) .. controls (389.62,232.38) and (391.28,232.38) .. (392.95,234.04) .. controls (394.62,235.71) and (396.28,235.71) .. (397.95,234.04) .. controls (399.62,232.37) and (401.28,232.37) .. (402.95,234.03) .. controls (404.62,235.7) and (406.28,235.7) .. (407.95,234.03) .. controls (409.62,232.36) and (411.28,232.36) .. (412.95,234.02) .. controls (414.62,235.68) and (416.28,235.68) .. (417.95,234.01) .. controls (419.62,232.34) and (421.28,232.34) .. (422.95,234.01) .. controls (424.62,235.67) and (426.28,235.67) .. (427.95,234) -- (430,234) -- (430,234) ;
\draw [color={rgb, 255:red, 245; green, 166; blue, 35 }  ,draw opacity=1 ][line width=2.25]    (132.9,72.48) -- (250,72.31) ;
\draw [color={rgb, 255:red, 126; green, 211; blue, 33 }  ,draw opacity=1 ][fill={rgb, 255:red, 126; green, 211; blue, 33 }  ,fill opacity=1 ][line width=2.25]    (250,72.31) -- (249.84,120.58) ;
\draw [color={rgb, 255:red, 126; green, 211; blue, 33 }  ,draw opacity=1 ][line width=2.25]    (290.64,132.57) -- (330,132.31) ;
\draw [color={rgb, 255:red, 208; green, 2; blue, 27 }  ,draw opacity=1 ][line width=2.25]    (330,132.31) -- (330,170.45) ;
\draw [color={rgb, 255:red, 74; green, 144; blue, 226 }  ,draw opacity=1 ][line width=2.25]    (389.67,232.65) -- (430,232.31) ;
\draw  [fill={rgb, 255:red, 0; green, 0; blue, 0 }  ,fill opacity=1 ] (324.42,132.31) .. controls (324.42,129.23) and (326.92,126.73) .. (330,126.73) .. controls (333.08,126.73) and (335.58,129.23) .. (335.58,132.31) .. controls (335.58,135.4) and (333.08,137.9) .. (330,137.9) .. controls (326.92,137.9) and (324.42,135.4) .. (324.42,132.31) -- cycle ;
\draw  [fill={rgb, 255:red, 0; green, 0; blue, 0 }  ,fill opacity=1 ] (244.42,72.31) .. controls (244.42,69.23) and (246.92,66.73) .. (250,66.73) .. controls (253.08,66.73) and (255.58,69.23) .. (255.58,72.31) .. controls (255.58,75.4) and (253.08,77.9) .. (250,77.9) .. controls (246.92,77.9) and (244.42,75.4) .. (244.42,72.31) -- cycle ;
\draw  [fill={rgb, 255:red, 0; green, 0; blue, 0 }  ,fill opacity=1 ] (424.42,232.31) .. controls (424.42,229.23) and (426.92,226.73) .. (430,226.73) .. controls (433.08,226.73) and (435.58,229.23) .. (435.58,232.31) .. controls (435.58,235.4) and (433.08,237.9) .. (430,237.9) .. controls (426.92,237.9) and (424.42,235.4) .. (424.42,232.31) -- cycle ;
\draw [color={rgb, 255:red, 74; green, 144; blue, 226 }  ,draw opacity=1 ][line width=2.25]  [dash pattern={on 2.53pt off 3.02pt}]  (138.61,338.38) -- (565.2,337.44) ;
\draw [color={rgb, 255:red, 126; green, 211; blue, 33 }  ,draw opacity=1 ][line width=2.25]  [dash pattern={on 2.53pt off 3.02pt}]  (138.28,379.88) -- (567.72,379.02) ;
\draw [color={rgb, 255:red, 245; green, 166; blue, 35 }  ,draw opacity=1 ][line width=2.25]  [dash pattern={on 2.53pt off 3.02pt}]  (138.28,399.88) -- (567.72,399.02) ;
\draw  [fill={rgb, 255:red, 248; green, 231; blue, 28 }  ,fill opacity=1 ] (236.57,273.34) .. controls (236.57,270.89) and (238.55,268.91) .. (241,268.91) .. controls (243.45,268.91) and (245.43,270.89) .. (245.43,273.34) .. controls (245.43,275.78) and (243.45,277.77) .. (241,277.77) .. controls (238.55,277.77) and (236.57,275.78) .. (236.57,273.34) -- cycle ;
\draw  [fill={rgb, 255:red, 248; green, 231; blue, 28 }  ,fill opacity=1 ] (256.57,273.34) .. controls (256.57,270.89) and (258.55,268.91) .. (261,268.91) .. controls (263.45,268.91) and (265.43,270.89) .. (265.43,273.34) .. controls (265.43,275.78) and (263.45,277.77) .. (261,277.77) .. controls (258.55,277.77) and (256.57,275.78) .. (256.57,273.34) -- cycle ;
\draw  [fill={rgb, 255:red, 248; green, 231; blue, 28 }  ,fill opacity=1 ] (276.57,273.34) .. controls (276.57,270.89) and (278.55,268.91) .. (281,268.91) .. controls (283.45,268.91) and (285.43,270.89) .. (285.43,273.34) .. controls (285.43,275.78) and (283.45,277.77) .. (281,277.77) .. controls (278.55,277.77) and (276.57,275.78) .. (276.57,273.34) -- cycle ;
\draw  [fill={rgb, 255:red, 248; green, 231; blue, 28 }  ,fill opacity=1 ] (316.57,273.34) .. controls (316.57,270.89) and (318.55,268.91) .. (321,268.91) .. controls (323.45,268.91) and (325.43,270.89) .. (325.43,273.34) .. controls (325.43,275.78) and (323.45,277.77) .. (321,277.77) .. controls (318.55,277.77) and (316.57,275.78) .. (316.57,273.34) -- cycle ;
\draw  [fill={rgb, 255:red, 248; green, 231; blue, 28 }  ,fill opacity=1 ] (296.57,273.34) .. controls (296.57,270.89) and (298.55,268.91) .. (301,268.91) .. controls (303.45,268.91) and (305.43,270.89) .. (305.43,273.34) .. controls (305.43,275.78) and (303.45,277.77) .. (301,277.77) .. controls (298.55,277.77) and (296.57,275.78) .. (296.57,273.34) -- cycle ;
\draw  [fill={rgb, 255:red, 248; green, 231; blue, 28 }  ,fill opacity=1 ] (336.57,273.34) .. controls (336.57,270.89) and (338.55,268.91) .. (341,268.91) .. controls (343.45,268.91) and (345.43,270.89) .. (345.43,273.34) .. controls (345.43,275.78) and (343.45,277.77) .. (341,277.77) .. controls (338.55,277.77) and (336.57,275.78) .. (336.57,273.34) -- cycle ;
\draw  [fill={rgb, 255:red, 248; green, 231; blue, 28 }  ,fill opacity=1 ] (356.57,273.34) .. controls (356.57,270.89) and (358.55,268.91) .. (361,268.91) .. controls (363.45,268.91) and (365.43,270.89) .. (365.43,273.34) .. controls (365.43,275.78) and (363.45,277.77) .. (361,277.77) .. controls (358.55,277.77) and (356.57,275.78) .. (356.57,273.34) -- cycle ;
\draw  [fill={rgb, 255:red, 248; green, 231; blue, 28 }  ,fill opacity=1 ] (256.57,293.34) .. controls (256.57,290.89) and (258.55,288.91) .. (261,288.91) .. controls (263.45,288.91) and (265.43,290.89) .. (265.43,293.34) .. controls (265.43,295.78) and (263.45,297.77) .. (261,297.77) .. controls (258.55,297.77) and (256.57,295.78) .. (256.57,293.34) -- cycle ;
\draw  [fill={rgb, 255:red, 248; green, 231; blue, 28 }  ,fill opacity=1 ] (236.57,293.34) .. controls (236.57,290.89) and (238.55,288.91) .. (241,288.91) .. controls (243.45,288.91) and (245.43,290.89) .. (245.43,293.34) .. controls (245.43,295.78) and (243.45,297.77) .. (241,297.77) .. controls (238.55,297.77) and (236.57,295.78) .. (236.57,293.34) -- cycle ;
\draw  [fill={rgb, 255:red, 248; green, 231; blue, 28 }  ,fill opacity=1 ] (276.57,293.34) .. controls (276.57,290.89) and (278.55,288.91) .. (281,288.91) .. controls (283.45,288.91) and (285.43,290.89) .. (285.43,293.34) .. controls (285.43,295.78) and (283.45,297.77) .. (281,297.77) .. controls (278.55,297.77) and (276.57,295.78) .. (276.57,293.34) -- cycle ;
\draw  [fill={rgb, 255:red, 248; green, 231; blue, 28 }  ,fill opacity=1 ] (296.57,293.34) .. controls (296.57,290.89) and (298.55,288.91) .. (301,288.91) .. controls (303.45,288.91) and (305.43,290.89) .. (305.43,293.34) .. controls (305.43,295.78) and (303.45,297.77) .. (301,297.77) .. controls (298.55,297.77) and (296.57,295.78) .. (296.57,293.34) -- cycle ;
\draw  [fill={rgb, 255:red, 248; green, 231; blue, 28 }  ,fill opacity=1 ] (336.57,293.34) .. controls (336.57,290.89) and (338.55,288.91) .. (341,288.91) .. controls (343.45,288.91) and (345.43,290.89) .. (345.43,293.34) .. controls (345.43,295.78) and (343.45,297.77) .. (341,297.77) .. controls (338.55,297.77) and (336.57,295.78) .. (336.57,293.34) -- cycle ;
\draw  [fill={rgb, 255:red, 248; green, 231; blue, 28 }  ,fill opacity=1 ] (316.57,293.34) .. controls (316.57,290.89) and (318.55,288.91) .. (321,288.91) .. controls (323.45,288.91) and (325.43,290.89) .. (325.43,293.34) .. controls (325.43,295.78) and (323.45,297.77) .. (321,297.77) .. controls (318.55,297.77) and (316.57,295.78) .. (316.57,293.34) -- cycle ;
\draw  [fill={rgb, 255:red, 248; green, 231; blue, 28 }  ,fill opacity=1 ] (356.57,293.34) .. controls (356.57,290.89) and (358.55,288.91) .. (361,288.91) .. controls (363.45,288.91) and (365.43,290.89) .. (365.43,293.34) .. controls (365.43,295.78) and (363.45,297.77) .. (361,297.77) .. controls (358.55,297.77) and (356.57,295.78) .. (356.57,293.34) -- cycle ;
\draw  [fill={rgb, 255:red, 248; green, 231; blue, 28 }  ,fill opacity=1 ] (376.57,293.34) .. controls (376.57,290.89) and (378.55,288.91) .. (381,288.91) .. controls (383.45,288.91) and (385.43,290.89) .. (385.43,293.34) .. controls (385.43,295.78) and (383.45,297.77) .. (381,297.77) .. controls (378.55,297.77) and (376.57,295.78) .. (376.57,293.34) -- cycle ;
\draw  [fill={rgb, 255:red, 248; green, 231; blue, 28 }  ,fill opacity=1 ] (276.57,313.34) .. controls (276.57,310.89) and (278.55,308.91) .. (281,308.91) .. controls (283.45,308.91) and (285.43,310.89) .. (285.43,313.34) .. controls (285.43,315.78) and (283.45,317.77) .. (281,317.77) .. controls (278.55,317.77) and (276.57,315.78) .. (276.57,313.34) -- cycle ;
\draw  [fill={rgb, 255:red, 248; green, 231; blue, 28 }  ,fill opacity=1 ] (256.57,313.34) .. controls (256.57,310.89) and (258.55,308.91) .. (261,308.91) .. controls (263.45,308.91) and (265.43,310.89) .. (265.43,313.34) .. controls (265.43,315.78) and (263.45,317.77) .. (261,317.77) .. controls (258.55,317.77) and (256.57,315.78) .. (256.57,313.34) -- cycle ;
\draw  [fill={rgb, 255:red, 248; green, 231; blue, 28 }  ,fill opacity=1 ] (296.57,313.34) .. controls (296.57,310.89) and (298.55,308.91) .. (301,308.91) .. controls (303.45,308.91) and (305.43,310.89) .. (305.43,313.34) .. controls (305.43,315.78) and (303.45,317.77) .. (301,317.77) .. controls (298.55,317.77) and (296.57,315.78) .. (296.57,313.34) -- cycle ;
\draw  [fill={rgb, 255:red, 248; green, 231; blue, 28 }  ,fill opacity=1 ] (316.57,313.34) .. controls (316.57,310.89) and (318.55,308.91) .. (321,308.91) .. controls (323.45,308.91) and (325.43,310.89) .. (325.43,313.34) .. controls (325.43,315.78) and (323.45,317.77) .. (321,317.77) .. controls (318.55,317.77) and (316.57,315.78) .. (316.57,313.34) -- cycle ;
\draw  [fill={rgb, 255:red, 248; green, 231; blue, 28 }  ,fill opacity=1 ] (356.57,313.34) .. controls (356.57,310.89) and (358.55,308.91) .. (361,308.91) .. controls (363.45,308.91) and (365.43,310.89) .. (365.43,313.34) .. controls (365.43,315.78) and (363.45,317.77) .. (361,317.77) .. controls (358.55,317.77) and (356.57,315.78) .. (356.57,313.34) -- cycle ;
\draw  [fill={rgb, 255:red, 248; green, 231; blue, 28 }  ,fill opacity=1 ] (336.57,313.34) .. controls (336.57,310.89) and (338.55,308.91) .. (341,308.91) .. controls (343.45,308.91) and (345.43,310.89) .. (345.43,313.34) .. controls (345.43,315.78) and (343.45,317.77) .. (341,317.77) .. controls (338.55,317.77) and (336.57,315.78) .. (336.57,313.34) -- cycle ;
\draw  [fill={rgb, 255:red, 248; green, 231; blue, 28 }  ,fill opacity=1 ] (376.57,313.34) .. controls (376.57,310.89) and (378.55,308.91) .. (381,308.91) .. controls (383.45,308.91) and (385.43,310.89) .. (385.43,313.34) .. controls (385.43,315.78) and (383.45,317.77) .. (381,317.77) .. controls (378.55,317.77) and (376.57,315.78) .. (376.57,313.34) -- cycle ;
\draw  [fill={rgb, 255:red, 248; green, 231; blue, 28 }  ,fill opacity=1 ] (396.57,313.34) .. controls (396.57,310.89) and (398.55,308.91) .. (401,308.91) .. controls (403.45,308.91) and (405.43,310.89) .. (405.43,313.34) .. controls (405.43,315.78) and (403.45,317.77) .. (401,317.77) .. controls (398.55,317.77) and (396.57,315.78) .. (396.57,313.34) -- cycle ;
\draw  [fill={rgb, 255:red, 248; green, 231; blue, 28 }  ,fill opacity=1 ] (396.57,293.34) .. controls (396.57,290.89) and (398.55,288.91) .. (401,288.91) .. controls (403.45,288.91) and (405.43,290.89) .. (405.43,293.34) .. controls (405.43,295.78) and (403.45,297.77) .. (401,297.77) .. controls (398.55,297.77) and (396.57,295.78) .. (396.57,293.34) -- cycle ;
\draw  [fill={rgb, 255:red, 248; green, 231; blue, 28 }  ,fill opacity=1 ] (376.57,273.34) .. controls (376.57,270.89) and (378.55,268.91) .. (381,268.91) .. controls (383.45,268.91) and (385.43,270.89) .. (385.43,273.34) .. controls (385.43,275.78) and (383.45,277.77) .. (381,277.77) .. controls (378.55,277.77) and (376.57,275.78) .. (376.57,273.34) -- cycle ;
\draw  [fill={rgb, 255:red, 248; green, 231; blue, 28 }  ,fill opacity=1 ] (396.57,273.34) .. controls (396.57,270.89) and (398.55,268.91) .. (401,268.91) .. controls (403.45,268.91) and (405.43,270.89) .. (405.43,273.34) .. controls (405.43,275.78) and (403.45,277.77) .. (401,277.77) .. controls (398.55,277.77) and (396.57,275.78) .. (396.57,273.34) -- cycle ;
\draw  [fill={rgb, 255:red, 248; green, 231; blue, 28 }  ,fill opacity=1 ] (237.07,313.6) .. controls (237.07,311.16) and (239.05,309.18) .. (241.5,309.18) .. controls (243.95,309.18) and (245.93,311.16) .. (245.93,313.6) .. controls (245.93,316.05) and (243.95,318.03) .. (241.5,318.03) .. controls (239.05,318.03) and (237.07,316.05) .. (237.07,313.6) -- cycle ;
\draw  [fill={rgb, 255:red, 248; green, 231; blue, 28 }  ,fill opacity=1 ] (276.07,253.34) .. controls (276.07,250.89) and (278.05,248.91) .. (280.5,248.91) .. controls (282.95,248.91) and (284.93,250.89) .. (284.93,253.34) .. controls (284.93,255.78) and (282.95,257.77) .. (280.5,257.77) .. controls (278.05,257.77) and (276.07,255.78) .. (276.07,253.34) -- cycle ;
\draw  [fill={rgb, 255:red, 248; green, 231; blue, 28 }  ,fill opacity=1 ] (256.07,253.34) .. controls (256.07,250.89) and (258.05,248.91) .. (260.5,248.91) .. controls (262.95,248.91) and (264.93,250.89) .. (264.93,253.34) .. controls (264.93,255.78) and (262.95,257.77) .. (260.5,257.77) .. controls (258.05,257.77) and (256.07,255.78) .. (256.07,253.34) -- cycle ;
\draw  [fill={rgb, 255:red, 248; green, 231; blue, 28 }  ,fill opacity=1 ] (296.07,253.34) .. controls (296.07,250.89) and (298.05,248.91) .. (300.5,248.91) .. controls (302.95,248.91) and (304.93,250.89) .. (304.93,253.34) .. controls (304.93,255.78) and (302.95,257.77) .. (300.5,257.77) .. controls (298.05,257.77) and (296.07,255.78) .. (296.07,253.34) -- cycle ;
\draw  [fill={rgb, 255:red, 248; green, 231; blue, 28 }  ,fill opacity=1 ] (316.07,253.34) .. controls (316.07,250.89) and (318.05,248.91) .. (320.5,248.91) .. controls (322.95,248.91) and (324.93,250.89) .. (324.93,253.34) .. controls (324.93,255.78) and (322.95,257.77) .. (320.5,257.77) .. controls (318.05,257.77) and (316.07,255.78) .. (316.07,253.34) -- cycle ;
\draw  [fill={rgb, 255:red, 248; green, 231; blue, 28 }  ,fill opacity=1 ] (356.07,253.34) .. controls (356.07,250.89) and (358.05,248.91) .. (360.5,248.91) .. controls (362.95,248.91) and (364.93,250.89) .. (364.93,253.34) .. controls (364.93,255.78) and (362.95,257.77) .. (360.5,257.77) .. controls (358.05,257.77) and (356.07,255.78) .. (356.07,253.34) -- cycle ;
\draw  [fill={rgb, 255:red, 248; green, 231; blue, 28 }  ,fill opacity=1 ] (336.07,253.34) .. controls (336.07,250.89) and (338.05,248.91) .. (340.5,248.91) .. controls (342.95,248.91) and (344.93,250.89) .. (344.93,253.34) .. controls (344.93,255.78) and (342.95,257.77) .. (340.5,257.77) .. controls (338.05,257.77) and (336.07,255.78) .. (336.07,253.34) -- cycle ;
\draw  [fill={rgb, 255:red, 248; green, 231; blue, 28 }  ,fill opacity=1 ] (376.07,253.34) .. controls (376.07,250.89) and (378.05,248.91) .. (380.5,248.91) .. controls (382.95,248.91) and (384.93,250.89) .. (384.93,253.34) .. controls (384.93,255.78) and (382.95,257.77) .. (380.5,257.77) .. controls (378.05,257.77) and (376.07,255.78) .. (376.07,253.34) -- cycle ;
\draw  [fill={rgb, 255:red, 248; green, 231; blue, 28 }  ,fill opacity=1 ] (396.07,253.34) .. controls (396.07,250.89) and (398.05,248.91) .. (400.5,248.91) .. controls (402.95,248.91) and (404.93,250.89) .. (404.93,253.34) .. controls (404.93,255.78) and (402.95,257.77) .. (400.5,257.77) .. controls (398.05,257.77) and (396.07,255.78) .. (396.07,253.34) -- cycle ;
\draw  [fill={rgb, 255:red, 248; green, 231; blue, 28 }  ,fill opacity=1 ] (237.07,254.1) .. controls (237.07,251.66) and (239.05,249.68) .. (241.5,249.68) .. controls (243.95,249.68) and (245.93,251.66) .. (245.93,254.1) .. controls (245.93,256.55) and (243.95,258.53) .. (241.5,258.53) .. controls (239.05,258.53) and (237.07,256.55) .. (237.07,254.1) -- cycle ;
\draw [color={rgb, 255:red, 248; green, 231; blue, 28 }  ,draw opacity=1 ][line width=4.5]    (429.79,278.24) -- (463.07,278.44) ;
\draw [color={rgb, 255:red, 248; green, 231; blue, 28 }  ,draw opacity=1 ][line width=4.5]    (463.07,278.44) -- (463.07,295) ;
\draw [color={rgb, 255:red, 126; green, 211; blue, 33 }  ,draw opacity=1 ][fill={rgb, 255:red, 126; green, 211; blue, 33 }  ,fill opacity=1 ][line width=2.25]    (249.84,120.58) -- (290.64,120.58) ;
\draw [color={rgb, 255:red, 126; green, 211; blue, 33 }  ,draw opacity=1 ][fill={rgb, 255:red, 126; green, 211; blue, 33 }  ,fill opacity=1 ][line width=2.25]    (290.64,120.58) -- (290.64,132.57) ;

\draw (477.07,271.53) node [anchor=north west][inner sep=0.75pt]    {$\mathcal{Q}^1$};
\draw (52,223.4) node [anchor=north west][inner sep=0.75pt]    {$(A^{1}[ b_{1}]_{j})_{j\leq a_{1}}$};
\draw (186.93,43.51) node [anchor=north west][inner sep=0.75pt]    {$\mathcal{Q}^K$};
\draw (207.13,92.11) node [anchor=north west][inner sep=0.75pt]    {$\mathcal{Q}^{K-1}$};
\draw (447,222.71) node [anchor=north west][inner sep=0.75pt]    {$( a_{1} ,b_{1})$};
\draw (332,173.85) node [anchor=north west][inner sep=0.75pt]  [font=\Huge]  {$\ddots $};
\draw (200,343.85) node [anchor=north west][inner sep=0.75pt]  [font=\Large]  {$\vdots $};
\draw (55,275.4) node [anchor=north west][inner sep=0.75pt]    {$( \wc \omega [ k]_{j})_{ j\leq a_{1}, k\in \{-1,\ \mydots ,\ b_{1} -1\}}$};

\end{tikzpicture}

\captionsetup{width=.8\linewidth}
\caption{An illustration of the base argument in the proof of Theorem \ref{ind_B}. The down-right path $\eta_1$ defined in \eqref{path1} is shown in yellow. The dual weights below $\eta_1$ are shown as yellow dots. }\label{step1fig}
\end{center}
\end{figure}
Clearly, $\mathcal{Q}^1$ is independent of  $\{\omega_{(j, n+1)}\}_{j \leq a_1, n \geq b_1}$. 
We will now rewrite the inter-departure time from \eqref{step1_before} in the following manner. For $k \in {2, \ldots, K}$ and $n \geq b_1$, we shall apply Proposition \ref{2level} to each pair of horizontal levels, namely $\{-1, 0\}, \{0,1\}, \mydots, \{b_1-1, b_1\}$. To simplify the notation, we use $A^1[s]$ to denote the $I[s]$ from \eqref{def_I} and $\wc \omega{}^1[s]$ to denote $\wc Y[s]$ from \eqref{shift}. We have
\begin{align}
&(D^{(n+k)}(A^k[-k], {\mydots}, A^1[-1], \omega[0], \omega[1], {\mydots} ,\omega[n])_j)_{j\leq a_1}\label{step1_calc}\\
&=(D^{(n+k)}(A^k[-k], {\mydots}, \wc \omega{}^1[-1], A^1[0],\omega[1], {\mydots} ,\omega[n])_j)_{j\leq a_1}\nonumber\\
& = (D^{(n+k)}(A^k[-k], {\mydots}, \wc \omega{}^1[-1], \wc \omega{}^1[0], A^1[1], {\mydots},\omega[n])_j)_{j\leq a_1}\nonumber\\
& \qquad  {\small \vdots}\nonumber\\
&  = (D^{(n+k)}(A^k[-k], {\mydots}, \wc \omega{}^1[-1], {\mydots}, \wc \omega{}^1[-b_1-1], A^1[b_1], {\mydots}, \omega[n])_j)_{j\leq a_1}\nonumber
\end{align}
Thus, we have turned the collection of random variables from  \eqref{step1_before} into the following 
\begin{equation}\label{step1}
\Big\{D^{(n+k)}(A^k[-k], {\mydots}, \wc \omega{}^1[-1], {\mydots}, \wc \omega{}^1[b_1-1], A^1[b_1], {\mydots} ,\omega[n])_j, \omega_{(j, n+1)} \Big\}_{j \leq a_1, k\in \{2, {\mydots} K\}, n \geq b_1}.
\end{equation}
Referring to Figure \ref{step1fig}, 
the positions of the increments
$(A^1[b_1])_{j\leq a_1}$ together with $\{\mathbf 
 y_{i}\}_{i_1 \leq i < \infty}$ form a down-right path 
\begin{equation}\label{path1}
{\eta}^1 = (-\infty, b_1) \to (a_1, b_1) \to \{\mathbf y_{i}\}_{i_1 \leq i < \infty}.
\end{equation}
Applying Proposition \ref{BK} to the last-passage value $G^{A[-1]^1}_{(0,-1), \bbullet}$ along the path $\eta_1$, we see that the increments in $\mathcal{Q}^1$ are independent, and they are independent of the random variables inside \eqref{step1}. Using the equality in \eqref{step1_calc}, we obtain that  the random variables in \eqref{step1_before} and $\mathcal{Q}^1$ are independent. Since $\mathcal{Q}^2, \mydots \mathcal{Q}^K$ are functions of \eqref{step1_before}, we have finished verifying the base case. The arguments hereafter for $\mathcal{Q}^2, {\mydots} \mathcal{Q}^K$ will be solely based on the random variables inside the collection \eqref{step1_before}. Hence, the independence with $\mathcal{Q}^1$ will not be affected. 

Similar to the widely recognized ``corner-flipping" induction technique for KPZ models, the inductive step remains exactly the same as the base case. And instead of arguing for $\mathcal{Q}^k$, we will work with  $\mathcal{Q}^2$ in order to explicitly show that the argument solely involves the random variables from \eqref{step1_before}. The general case for $\mathcal{Q}^k$ will follow the same arguments as $\mathcal{Q}^2$, with the only variation being the changes to the indices.

To start, we note again that for computing $\mathcal{Q}^2, \mydots \mathcal{Q}^K$, we only need the random variables inside \eqref{step1_before}, hence the independence with $\mathcal{Q}^1$ will not be affected.
Let increments inside $\mathcal{Q}^2$ be computed using \eqref{step1_before}, and we note that $\mathcal{Q}^3, \mydots \mathcal{Q}^K$ are functions of the following random variables 
\begin{equation}\label{step2_before}\Big\{(D^{(n+k)}(A^k[-k], {\mydots},A^2[-2],  A^1[-1], \omega[0], {\mydots}, \omega[n])_j, \omega_{(j,n+1)}\Big\}_{j \leq a_2, k \in \{3, {\mydots} K\}, n \geq b_2}.
\end{equation}
Just like before, we will apply Proposition \ref{2level} to the departure process in \eqref{step2_before} for the pairs of levels $\{-2,-1\}$, $\{-1,0\}$, $\mydots$, $\{b_2-1, b_2\}$. For each $k\in \{3, {\mydots}, K\}$, $n\geq b_2$, 
\begin{align}
&(D^{(n+k)}(A^k[-k], {\mydots}, A^2[-2], A^1[-1], \omega[0], {\mydots} ,\omega[n])_j)_{j\leq a_2}\label{step2_calc}\\
&  = (D^{(n+k)}(A^k[-k], {\mydots},  \wc A{}^{1,2}[-2], A^2[-1], \omega[0],  {\mydots},  \omega[n])_j)_{j\leq a_2} \nonumber\\
&= (D^{(n+k)}(A^k[-k], {\mydots},  \wc A{}^{1,2}[-2], \wc\omega{}^2 [-1], A^2[0],  {\mydots},  \omega[n])_j)_{j\leq a_2} \nonumber\\
& \qquad  {\small \vdots}\nonumber\\
&  = (D^{(n+k)}A^k[-k], {\mydots},  \wc A{}^{1,2}[-2], \wc\omega{}^2 [-1] ,{\mydots}, \wc\omega{}^2[b_2-1], A^2[b_2] ,{\mydots}, \omega[n])_j)_{j\leq a_2}.\nonumber
\end{align}

The increments inside $(A^2[b_2])_{j\leq a_2}$ and $\{\mathbf y_i\}_{i_2\leq i \leq i_1}$ form a down-right path
$$\eta^2 = (-\infty, b_2) \to (a_2, b_2) \to \{\mathbf y_i\}_{i_2 \leq i \leq i_1}.$$
Similar to before, applying Proposition \ref{BK} to $G^{A^2[-2]} _{(0,-2),\bbullet}$ along the down-right paths $\eta^2$, we obtain that the increments inside $\mathcal{Q}^2$ are independent, furthermore, they are independent of the random variables from
$$
\Big\{D(^{(n+k)}A^k[-k], {\mydots},  A^{1,2}[-2], \wc\omega{}^2 [-1] {\mydots}, \wc\omega{}^2[b_2-1], A^2[b_2], {\mydots}, \omega[n])_j, \omega_{(j,n+1)}\Big\}_{j\leq a_2, k \in \{3,{\mydots}, K\}, n\geq b_2}.
$$
Hence, by the equalities from \eqref{step2_calc}, $\mathcal{Q}^2$ is independent with \eqref{step2_before}.
On the other hand, $\mathcal{Q}^3, {\mydots} \mathcal{Q}^K$ are calculated using of \eqref{step2_before}, 
we have shown that the increments inside 
$\mathcal{Q}^1 \cup \mathcal{Q}^2$ are independent, and they are independent of $\sigma(\mathcal{Q}^3, {\mydots} \mathcal{Q}^K)$. 

For $\mathcal{Q}^k$, the argument is the same as $\mathcal{Q}^2$,  except for the changes in the indices. With this, we have finished the proof of Theorem \ref{ind_B}.
\end{proof}

Next, we come back to the CGM.  In order to make a connection to the independence result in Theorem \ref{ind_B}, we rotate the direction in our CGM by $180^\circ$. Thus, now the admissible paths take $-\mathbf e_1$ and $-\mathbf e_2$ steps, and semi-infinite geodesics have directions in $-\mathbb{R}^2_{> 0} \setminus \{(0,0)\}$. Let $\mathbf v_N$ be a sequence of vectors such that $|\mathbf v_N|_1 \to \infty$ and $\mathbf v_N/|\mathbf v_N|_1 \to -\boldsymbol\xi[\rho]/|\boldsymbol\xi[\rho]|_1$, the Busemann function in the direction $-\boldsymbol\xi[\rho]$ is then defined as the almost sure limit
 \begin{equation}\label{neg_dir}\wt B^\rho_{\mathbf x,\mathbf y} = \lim_{N\to \infty} \wt G_{\mathbf x, \mathbf v_N} - \wt G_{\mathbf y, \mathbf v_N}.
 \end{equation}

Fix $K \in \mathbb{Z}_{>0}$ and $\rho_1 < \mydots < \rho_K \in (0, 1)$. Define an infinite down-right path $\mathcal{Y} =  \{\mathbf y_i\}_{i\in \mathbb{Z}}$ with $\mathbf y_{i+1} - \mathbf y_i \in \{\mathbf e_1, -\mathbf e_2\}$. Fix any $i_1 > i_2> \mydots > i_{K-1}$, and define the collections of Busemann increments
\begin{align}
\mathcal{E}^1 &= \{\wt B^{\rho_1}_{\mathbf y_{i+1}, \mathbf y_i} : i_1 \leq i <\infty  \}\nonumber\\
\mathcal{E}^{k} &= \{\wt B^{\rho_{k}}_{\mathbf y_{i+1}, \mathbf y_i} : i_{k}\leq  i < i_{k-1} \} \quad \text{ for }k \in \{2, \mydots, K-1\}\nonumber\\
\mathcal{E}^K &= \{\wt B^{\rho_K}_{\mathbf y_{i+1}, \mathbf y_i} : -\infty < i <i_{K-1} \}. \nonumber
\end{align}
Then, clearly the following proposition implies Theorem \ref{main}.
\begin{proposition}\label{last_step}
The collection of random variables $\mathcal{E}^1 \cup \mathcal{E}^2 \cup {\mydots} \cup \mathcal{E}^K$ are mutually independent. 
\end{proposition}
\begin{proof}
Let us fix the same down-right paths $\mathcal{Y}$, indices $i_1>i_2 >\mydots > i_{K-1}$, and define $\{\mathcal{E}_i\}_{i=1}^K$ and $\{\mathcal{Q}_i\}_{i=1}^K$. We will show that the collection of random variables $\mathcal{E}^1 \cup \mathcal{E}^2 \cup {\mydots} \cup \mathcal{E}^K$ (viewed as a random vector) is equal in distribution with $\mathcal{Q}^1 \cup \mathcal{Q}^{2} \cup {\mydots} \cup \mathcal{Q}^K$. Then, Theorem \ref{ind_B} gives us the desired independence as claimed in our proposition. 

The equality in distribution directly follows from two results from \cite{Fan-Sep-20}. Let $\wt B^{\rho_{k}}[m]$ to denote the collection $(\wt B^{\rho_{k}}_{(j,m), (j-1,m)})_{j\in \mathbb{Z}}$.  Theorem 3.2 from \cite{Fan-Sep-20} states that the joint distribution of 
$(\wt B^{\rho_1}[-1], \wt B^{\rho_{2}}[-1], \mydots, \wt B^{\rho_{K}}[-1])$ is equal to 
$$\Big(A^1[-1], D^{(1)}(A^2[-2], A^1[-1]), \mydots , D^{(K-1)}(A^{K}[-K], \mydots, A^1[-1])\Big).$$
Then, Lemma 3.3 of \cite{Fan-Sep-20} states that the Busemann increments in each direction $\rho_k$ above the line $y=-1$ has the same distribution as the increments from $G_{(0,-1),\bbullet}^{D^{(k-1)}(A^k[-k], \mydots, A^1[-1])}$, 
which are equal to the increments $G_{(0,-k),\bbullet}^{A^k[-k]}$ by Proposition \ref{A_nest}. With this, we have finished the proof of this proposition.

\end{proof}

\section{Modifications for the inverse-gamma polymer }\label{ig_poly}

In this section, we start by outlining the necessary modifications required to adapt the proofs from Section \ref{q_basic} and Section \ref{pf_main1} for the inverse-gamma polymer. In the second part, we will provide a proof of Theorem \ref{mid}.

From Section \ref{q_basic}, the LPP with horizontal boundary defined in \eqref{Gh}, will be replaced by the partition function starting from a horizontal boundary. Given the boundary values $Y[0]$, define
$$h_i = \begin{cases}
\prod_{j=0}^i Y_{(j,0)} \qquad & \textup{ if } i \geq 0\\
\prod_{j=1}^{|i|} Y_{(-j,0)}^{-1} \qquad & \textup{ if } i < 0.
\end{cases}
$$
The partition function with horizontal boundary $Y[0]$ is defined by 
\begin{equation}\label{free}
\log Z^{Y[0]}_{\mathbf 0,\mathbf x} =  \log \Big(\sum_{i \in \mathbb{Z}} h_i \cdot Z_{(i,0), \mathbf x}\Big).
\end{equation}
Then the increments $I, J$ and the dual weights from \eqref{def_inc} are now given by 
\begin{align*}
I_{[\![\mathbf z, \mathbf z+\mathbf e_1]\!]}  &= \begin{cases}
Y_{\mathbf z+\mathbf e_1}  \qquad & \text{ if }\mathbf z\cdot \mathbf e_2 =0\\
 Z^{{Y[0]}} _{\mathbf 0,\mathbf z+\mathbf e_1} /Z^{{Y[0]}}_{\mathbf 0,\mathbf z} \qquad & \text{ if } \mathbf z\cdot \mathbf e_2 \geq 1
\end{cases}\\ 
J_{[\![\mathbf z, \mathbf z+\mathbf e_2]\!]} &= Z^{{Y[0]}} _{\mathbf 0, \mathbf z+\mathbf e_2} / Z^{Y[0]} _{\mathbf 0,\mathbf z} \\ 
\wc Y_{\mathbf z} &= \frac{1}{{I_{[\![\mathbf z, \mathbf z+\mathbf e_1]\!]}} + {J_{[\![\mathbf z, \mathbf z+\mathbf e_2]\!]}}}.
\end{align*}
The inter-departure time and the sojourn time become 
\begin{align}
D^{(1)}(Y[0], Y[1]) &= \big(Z^{Y[0]} _{\mathbf 0, (j,1)}/ Z^{Y[0]} _{\mathbf 0, (j-1,1)}\big)_{j\in \mathbb{Z}} \nonumber\\
S(Y[0], Y[1]) &= \big(Z^{Y[0]} _{\mathbf 0, (j,1)} /Z^{Y[0]} _{\mathbf 0, (j,0)}\big)_{j\in \mathbb{Z}}.\nonumber
\end{align}

Next, we will give the modifications needed in order to lift the results presented in Section \ref{q_basic} to the positive-temperature setting. The main ingredient for proving Proposition \ref{A_nest} was  \cite[Lemma A.4]{Fan-Sep-20}, but now we replace it with \cite[Lemma 6.3]{bat-fan-sep-23-}. As a result, Proposition \ref{D_to_LPP} and Proposition \ref{S_to_LPP} remain valid, as they follow from Proposition \ref{A_nest}.
The queueing identity in Proposition \ref{2level} relied on \cite[Lemma 4.5]{Fan-Sep-20}, and it can be replaced with the positive-temperature version in  \cite[Lemma A.5]{Bus-Sep-22}.
Regarding the independence result in Proposition \ref{BK}, it was originally based on  \cite[Lemma B.2]{Fan-Sep-20}. The corresponding positive-temperature counterpart can be found in \cite[Lemma B.2]{Bus-Sep-22} and applied to our induction argument.

Next, let us move on to Section \ref{pf_main1}. To start, the difference of last-passage values inside the collection $\{\mathcal{Q}^k\}_{k=1}^K$ will be replaced by the corresponding ratios of partition functions. The statement and proof of Theorem \ref{ind_B} remain unchanged for the positive-temperature setting.
In Proposition \ref{last_step}, the random variables inside the collection $\{\mathcal{E}_k\}_{k=1}^K$ will now be replaced by the exponential of the Busemann increments, i.e.~$\exp({B^{\rho_{k}}_{\mathbf y_{i+1}, \mathbf y_i}})$. 
The main inputs for the proof rely on Theorem 3.2 and Lemma 3.3 from \cite{Fan-Sep-20}. The positive-temperature version of these can be found in \cite[Theorem 4.1]{bat-fan-sep-23-}, \cite[Theorem 6.23(a)]{bat-fan-sep-23-} and step 2 in the proof of Theorem 6.23 in \cite{bat-fan-sep-23-}.

\subsection{Proof of Theorem \ref{mid}}

Instead of proving the expectation bound in Theorem \ref{mid},  we will establish the following tail estimate, which applies to more general endpoints. This result will directly imply Theorem \ref{mid}. 

Recall that $Z^\textup{ptl}_{\mathbf 0, N}$ and $Q^\textup{ptl}_{\mathbf 0, N}$ denote the point-to-line partition function and the quenched polymer measure between the origin and the line $x+y = 2N$. Let $\{\textsf{end} \leq \delta N^{2/3}\}$ denote the collection of directed paths that start from the origin and end at $(N+s, N-s)$ for $|s| \leq \delta N^{2/3}$. Let us also introduce a new notation.
Fix $m\in \mathbb{Z}$, let $\{\textsf{end}_m^+ \leq \delta N^{2/3}\}$ denote the collection of directed paths starting from the origin and ending at $(N+m+s, N-m-s)$ for $0\leq s \leq \delta N^{2/3}$.

\begin{proposition} \label{prop_mid}
There exist finite strictly positive constants $\delta_0, N_0, C$ such that for each $0< \delta \leq \delta_0$, $N\geq N_0$, and $|m| \leq |\log \delta|N^{2/3}$,
$$\mathbb{P}\Big(Q^{\textup{ptl}}_{\mathbf 0,N}\{\textup{\textsf{end}}_m^+\leq \delta N^{2/3}\}  \geq e^{-|\log \delta|^2 \sqrt{\delta}N^{1/3}}\Big) \leq C|\log (\delta\vee N^{-2/3})|^{10} (\delta\vee N^{-2/3}).$$
\end{proposition}

\begin{proof}
Note it suffices for us to give the proof assuming that $\delta \geq \tfrac{1}{2} N^{-2/3}$.
To simplify the notation, let $r = |\log \delta|$. Then, it holds that 
\begin{align}
&\mathbb{P}\Big(Q^{\textup{ptl}}_{\mathbf 0,N}\{\textup{\textsf{end}}_m^+\leq \delta N^{2/3}\}  \geq e^{-|\log \delta|^2 \sqrt{\delta}N^{1/3}}\Big) \nonumber\\
& =  \mathbb{P}\Big(\log Z^{\textup{ptl}}_{\mathbf 0, N} -  \log Z^{\textup{ptl}}_{\mathbf 0, N}\{\textsf{end}_m^+\leq \delta N^{2/3}\}  \leq |\log \delta|^2 \sqrt{\delta}N^{1/3}\Big)\nonumber\\
& \leq  \mathbb{P}\Big(\log Z^{\textup{ptl}}_{\mathbf 0, N} \{\textsf{end} \leq 2rN^{2/3}\}-  \log Z^{\textup{ptl}}_{\mathbf 0, N}\{\textsf{end}_m^+\leq \delta N^{2/3}\}  \leq |\log \delta|^2 \sqrt{\delta}N^{1/3}\Big)\\
& \leq  \mathbb{P}\Big(\max_{|i| \leq \floor{2rN^{2/3}}}\Big[\log Z_{\mathbf 0, (N+i, N-i)} \Big] - \max_{0 \leq j \leq {\delta N^{2/3}}} \Big[\log Z_{\mathbf 0, (N+m+j, N-m-j)} \Big]  \leq 2|\log \delta|^2 \sqrt{\delta}N^{1/3}\Big)\nonumber\\
& \leq  \mathbb{P}\Big(\max_{|i| \leq {2rN^{2/3}}}\Big[\log\frac{ Z_{\mathbf 0, (N+i, N-i)}}{Z_{\mathbf 0, (N+m, N-m)}} \Big] - \max_{0 \leq j \leq {\delta N^{2/3}}} \Big[\log \frac{Z_{\mathbf 0, (N+m+j, N-m-j)}}{{Z_{\mathbf 0, (N+m, N-m)}}} \Big]  \leq 2|\log \delta|^2 \sqrt{\delta}N^{1/3}\Big)\nonumber\\
& \leq  \mathbb{P}\Big(\max_{|i| \leq {2rN^{2/3}}}\Big[\log \frac{Z_{\mathbf 0, (N+i, N-i)}}{Z_{\mathbf 0, (N+m, N-m)}} \Big]  \leq 3|\log \delta|^2  \sqrt{\delta}N^{1/3}\Big) \nonumber\\ 
& \qquad \qquad \qquad  + \mathbb{P}\Big( \max_{0 \leq j \leq {\delta N^{2/3}}} \Big[\log\frac{ Z_{\mathbf 0, (N+m+j, N-m-j)}}{{Z_{\mathbf 0, (N+m, N-m)}}} \Big]  \geq |\log \delta|^2 \sqrt{\delta}N^{1/3}\Big)\nonumber\\
& \leq  \mathbb{P}\Big(\Big\{\max_{1\leq k \leq rN^{2/3} }\Big[\log \frac{Z_{\mathbf 0, (N+m+k, N-m-k)}}{Z_{\mathbf 0, (N+m, N-m)}} \Big]  \leq 3|\log \delta|^2  \sqrt{\delta}N^{1/3} \} \nonumber\\
& \qquad \qquad \qquad \bigcap \Big\{\max_{-rN^{2/3}\leq k \leq -1 }\Big[\log \frac{Z_{\mathbf 0, (N+m+k, N-m-k)}}{Z_{\mathbf 0, (N+m, N-m)}} \Big]  \leq 3|\log \delta|^2  \sqrt{\delta}N^{1/3}\Big\}\Big) \label{trest1}\\ 
& \qquad \qquad \qquad  \qquad \qquad + \mathbb{P}\Big( \max_{0 \leq j \leq {\delta N^{2/3}}} \Big[\log\frac{ Z_{\mathbf 0, (N+m+j, N-m-j)}}{{Z_{\mathbf 0, (N+m, N-m)}}} \Big]  \geq |\log \delta|^2 \sqrt{\delta}N^{1/3}\Big).\label{trest2}
\end{align}

Next, we will bound the probabilities \eqref{trest1} and \eqref{trest2} separately, using a random walk comparison. To do this, we will use Theorem 3.28 from \cite{bas-sep-she-23} for the endpoint $(N,N)$. Note that the estimates from there are applicable to our shifted endpoint $(N+m, N-m)$ because we have assumed that $|m| \leq rN^{2/3}$. 
Theorem 3.28 from \cite{bas-sep-she-23} states that there exists a fixed $q_0$ and two collections of i.i.d.~random variables $\{H_i\}_{i\in \mathbb{Z}}, \{V_i\}_{i\in \mathbb{Z}}$ such that the following holds. 
The distributions of
\begin{equation}\label{step_dist}
H_i \text{ and} -V_i \stackrel{d}{=} \log X- \log Y
\end{equation}
where $X \sim \textup{Ga}^{-1}(1/2-q_0 r N^{-1/3})$, $Y\sim  \textup{Ga}^{-1}(1/2 + q_0 r N^{-1/3})$ and  $X, Y$ are independent. Furthermore, there exists an an event $A$ with $\mathbb{P}(A) \geq 1-e^{-Cr^3}$
such that on the event $A$, for each $ 1\leq k \leq rN^{2/3}$
\begin{align}\label{bd1}
\log \tfrac{9}{10} + \sum_{i=1}^k V_i \leq \log \frac{Z_{\mathbf 0,(N+m+k, N-m-k)}}{Z_{\mathbf 0, (N+m, N-m)}} \leq \log \tfrac{10}{9} + \sum_{i=1}^k H_i,
\end{align}
and for each $-rN^{2/3} \leq k \leq -1$, 
\begin{align}\label{bd2}
\log \tfrac{9}{10} + \sum_{i=k+1}^{0} -H_i \leq \log \frac{Z_{\mathbf 0,(N+m+k, N-m-k)}}{Z_{\mathbf 0, (N+m, N-m)}} \leq \log \tfrac{10}{9} + \sum_{i=k+1}^{0} -V_i.
\end{align}
Next, we will introduce a coupling for $\{H_i\}_{i\in \mathbb{Z}}, \{V_i\}_{i\in \mathbb{Z}}$ using the Busemann functions. 

Recall in the proof of Theorem 3.28 from \cite{bas-sep-she-23}, we attached two different (southwest) boundary weights, call them $W$ and $R$, to the southwest boundary based at the vertex $(-1,-1)$. 
More precisely, define $W_{(-1,-1)} = R_{(-1,-1)} = 0$, and for $k \geq 0$, the boundary $W$ is given by the independent collection random variables with the distributions
\begin{align*}
W_{(k, -1)} &\sim \textup{Ga}^{-1}(1/2 -q_0 rN^{-1/3})\\
W_{(-1, k)} &\sim \textup{Ga}^{-1}(1/2 +q_0 rN^{-1/3}).
\end{align*}
And the boundary $R$ is given by the independent collection 
\begin{align*}
R_{(k, -1)} &\sim \textup{Ga}^{-1}(1/2 +q_0 rN^{-1/3})\\
R_{(-1, k)} &\sim \textup{Ga}^{-1}(1/2 -q_0 rN^{-1/3}).
\end{align*}
Then, the random variables $\{H_i\}$ and $\{V_i\}$ from the estimates \eqref{bd1} and \eqref{bd2} are computed as the increments of the free energies with these boundaries  
\begin{align*}
H_i = \log Z^W_{(-1,-1), (N+m+i+1, N-m-i-1)} - \log Z^W_{(-1,-1), (N+m+i, N-m-i))}\\
V_i = \log Z^R_{(-1,-1), (N+m+i+1, N-m-i-1)} - \log Z^R_{(-1,-1), (N+m+i, N-m-i))}.
\end{align*}
Note that the free energies mentioned above have southwest boundaries instead of horizontal boundaries. They are defined the same way as the point-to-point free energy introduced below \eqref{def_part}, except that the random weights are no longer i.i.d.~due to the presence of the boundary weights.

In contrast to Theorem 3.28 from \cite{bas-sep-she-23}, where no specific coupling between $W$ and $R$ was required, we will now proceed to establish a specific coupling using the Busemann function. This coupling will, in turn, provide us with a coupling between the $H_i$'s and $V_i's$. Such coupling has appeared previously for the CGM and the inverse-gamma polymer, see \cite{balzs2019nonexistence, balzs2020local, Bus-Sep-22, ras-sep-she-, seppcoal}.

Recall $\wt B^\rho_{x,y}$ is the Busemann function defined in the southwest direction $-\xi[\rho]$. We now fix our two boundaries to be
\begin{align*}
W_{(k, -1), (k+1, -1)} &= e^{\wt B^{1/2 + q_0 r N^{-1/3}}_{(k+1, -1), (k, -1)}}, \quad 
W_{(-1, k), (-1, k+1)} = e^{\wt B^{1/2 + q_0 r N^{-1/3}}_{(-1, k+1), (-1, k)}}\\
R_{(k, -1), (k+1, -1)} &= e^{\wt B^{1/2 - q_0 r N^{-1/3}}_{(k+1, -1), (k, -1)}},
\quad 
R_{(-1, k), (-1, k+1)} = e^{\wt B^{1/2 - q_0 r N^{-1/3}}_{(-1, k+1), (-1, k)}}.
\end{align*}
Next, we will leverage the same fact used in the proof of Proposition \ref{last_step}, which essentially states that the increment computed from free energy with the Busemann boundary, $W$ and $R$,  coincides precisely with the Busemann increments defined in \eqref{neg_dir}, i.e.
\begin{align}
H_i &= \log Z^W_{(-1,-1), (N+m+i+1, N-m-i-1)} - \log Z^W_{(-1,-1), (N+m+i, N-m-i))} \nonumber\\
& = {\wt B^{1/2 + q_0 r N^{-1/3}}_{N+m+i+1, N-m-i-1), (N+m+i, N-m-i)}}\label{buseH}\\
V_i &= \log Z^R_{(-1,-1), (N+m+i+1, N-m-i-1)} - \log Z^R_{(-1,-1), (N+m+i, N-m-i))} \nonumber\\
&= {\wt B^{1/2 - q_0 r N^{-1/3}}_{(N+m+i+1, N-m-i-1), (N+m+i, N-m-i)}}.\label{buseV}
\end{align}

We continue to bound the probabilities \eqref{trest1} and \eqref{trest2}.  Starting with \eqref{trest1}, the first inequality below follows from estimates \eqref{bd1} and \eqref{bd2}; the next equality follows from identities \eqref{buseH}, \eqref{buseV} and the independence result from Theorem \ref{main}, 
\begin{align*}
\eqref{trest1} &\leq  \mathbb{P}\Big(\Big\{\max_{1\leq k \leq rN^{2/3} } \sum_{i=1}^k V_i  \leq 4|\log \delta|^2  \sqrt{\delta}N^{1/3} \Big\} \nonumber\\
& \qquad \qquad \qquad \bigcap \Big\{\max_{-rN^{2/3}\leq k \leq -1 } \sum_{i={k+1}}^0 -H_i \leq 4|\log \delta|^2  \sqrt{\delta}N^{1/3}\Big\}\Big)  + \mathbb{P}(A^c) \\
& = \mathbb{P}\Big(\max_{1\leq k \leq rN^{2/3} } \sum_{i=1}^k V_i  \leq 4|\log \delta|^2  \sqrt{\delta}N^{1/3} \Big) \nonumber\\
& \qquad \qquad \qquad \cdot \mathbb{P}\Big(\max_{-rN^{2/3}\leq k \leq -1} \sum_{i={k+1}}^0 -H_i \leq 4|\log \delta|^2  \sqrt{\delta}N^{1/3}\Big)  + \mathbb{P}(A^c) 
\end{align*}
The probability $\mathbb{P}(A^c) \leq e^{-Cr^3} \leq \delta$ when $\delta_0$ if fixed sufficiently small. The distributions of $V_i$ and $H_i$ are given explicitly in \eqref{step_dist}, then the product of the two probabilities for the running maximum can be bounded by $C|\log \delta|^6 \delta$ using Proposition \ref{rwest}. With this, we have finished the estimate for \eqref{trest1}.

For \eqref{trest2}, estimate \eqref{bd1} gives the first inequality below, the second inequality follows from Theorem \ref{max_sub_exp} and the verification of $H_i$ being a sub-exponential random variable is presented in Proposition \ref{Ga_sub_exp},
$$\eqref{trest2} \leq \mathbb{P}\Big( \max_{1 \leq j \leq {\delta N^{2/3}}} \sum_{i=1}^j H_i  \geq \frac{1}{2}|\log \delta|^2 \sqrt{\delta}N^{1/3}\Big) \leq e^{-C|\log \delta|^2} \leq \delta.$$
With this, we have given  the estimate for \eqref{trest2} which completes the proof of the theorem.

\end{proof}

Finally, we give a quick proof for Theorem \ref{mid}.

\begin{proof}[Proof of Theorem \ref{mid}]
Again, we may assume that $\delta \geq \tfrac{1}{2} N^{-2/3}$. By definition 
$$\{\textup{\textsf{end}}\leq \delta N^{2/3}\} =  \{\textup{\textsf{end}}_{-\delta N^{2/3}}^+\leq 2\delta N^{2/3}\}.$$
Let us define the event 
$A = \big\{Q^{\textup{ptl}}_{\mathbf 0,N}\{\textup{\textsf{end}}_m^+\leq \delta N^{2/3}\}  \geq e^{-|\log \delta|^2 \sqrt{\delta}N^{1/3}} \big\},$
which is the event inside the probability measure from Proposition \ref{prop_mid}. 
Then, we have the following expectation bound
\begin{align*}
\mathbb{E}\Big[Q^{\textup{ptl}}_{\mathbf 0, N}\{\textup{\textsf{end}} \leq \delta N^{2/3}\} \Big]  &= \mathbb{E}\Big[ Q^{\textup{ptl}}_{\mathbf 0, N}\{\textup{\textsf{end}}_{-\delta N^{2/3}}^+\leq 2\delta N^{2/3}\} \Big]\\
&=  \mathbb{E}\Big[ Q^{\textup{ptl}}_{\mathbf 0, N}\{\textup{\textsf{end}}_{-\delta N^{2/3}}^+\leq 2\delta N^{2/3}\} \mathbbm{1}_A\Big] + \mathbb{E}\Big[Q^{\textup{ptl}}_{\mathbf 0, N} \{\textup{\textsf{end}}_{-\delta N^{2/3}}^+\leq 2\delta N^{2/3}\}  \mathbbm{1}_{A^c}\Big]\\
& \leq \mathbb{P}(A) + e^{-|\log \delta|^2 \sqrt{\delta}N^{1/3}}\\
& \leq C |\log \delta|^{10} \delta \quad \text{ by Proposition \ref{prop_mid}}.
\end{align*}

\end{proof}

\section{Results for the semi-discrete and continuous models}\label{more_models}

Much like the modification discussed in the previous section for the inverse-gamma polymer, our proof readily extends to the Brownian LPP and the O'Connell-Yor polymer. This extension is feasible due to recent advancements in the queueing description of the joint distribution of the Busemann process for these semi-discrete models \cite{groathouse2023jointly, 2021arXiv211210729S}.

For the two continuous models, the KPZ equation and the directed landscape, part of our results from the semi-discrete models directly carry over. This is because, along a horizontal line, the finite-dimensional distributions of their Busemann process are equal to the ones from the O'Connell-Yor polymer and Brownian LPP, respectively, after suitable reflection and scaling \cite{2021arXiv211210729S,groathouse2023jointly}.

We start by defining Busemann functions for the Brownian LPP and the O'Connell-Yor polymer.
Let $\{W_k(\bbullet)\}_{k \in \mathbb{Z}}$ be a collection of independent, two-sided Brownian motions.
For $m,n \in \mathbb{Z}$ and $x,y \in \mathbb{R}$ with $m< n$ and $x\leq y$, defined the collection of up-right paths from $(x,m)$ to $(y,n)$ to be 
$$\Pi^{(x,m), (y,n)} = \{(s_{m-1}, s_m, \dots, s_n) \in \mathbb{R}^{n-m+2} : x= s_{m-1} \leq s_m \leq \dots \leq s_n = y\}.$$
The (Brownian) last-passage value is defined as 
$$L_{(x,m), (y,n)} = \sup\Big\{\sum_{k=m}^n B_k(s_k) - B_k(s_{k-1}) : (s_{m-1}, s_m, \dots, s_n)  \in\Pi^{(x,m), (y,n)}  \Big\},$$
and for $m=n$, we set $L_{(x,m), (y,n)} =  B_m(y) - B_m(x)$.

The O'Connell-Yor polymer is the positive-temperature analog of the Brownian LPP. For $m,n \in \mathbb{Z}$ and $x,y \in \mathbb{R}$ with $m< n$ and $x\leq y$, define the point-to-point partition function to be 
$$Z_{(x,m), (y,n)} = \int e^{\sum_{k=m}^n B_k(s_k) - B_k(s_{k-1})} \mathbbm{1}_{\{x= s_{m-1} \leq s_m \leq \dots \leq s_n =y\}} ds_{m} \dots ds_{n-1},$$
and for $m=n$, we set  $Z_{(x,m), (y,n)} =  e^{B_m(y) - B_m(x)}$.

For $\mathbf x, \mathbf y \in \mathbb{Z} \times \mathbb {R}$ and $\theta\in \mathbb{R}_{>0}$, the Busemann functions for the Brownian LPP and the O'Connell-Yor polymer, in the direction parameterised by $\theta$, are defined to be the following almost sure limits
\begin{align*}
B^\theta_{\mathbf{x}, \mathbf{y}} &= \lim_{n\to \infty} L_{\mathbf{x}, (t_n, n)} -  L_{\mathbf{y}, ( t_n, n)},\\
B^\theta_{\mathbf{x}, \mathbf{y}} &= \lim_{n\to \infty} \log Z_{\mathbf{x}, (t_n, n )} -  \log Z_{\mathbf{y}, (t_n, n)}
\end{align*}
where $\{t_n\}$ is any sequence of real numbers such that $\lim_{n\to \infty} \frac{t_n}{n} = \theta$.
For the existence of these limits, we refer to \cite{Alb-Ras-Sim-20}.

With these definitions, we are ready to state our independence results for the semi-discrete models. 
Let us fix a sequence of directions $\theta_1 \geq \theta_2 \geq \dots  \geq \theta_n>0$, and a sequence of points $\mathbf y_0, \dots, \mathbf y_n$ along a down-right path in $\mathbb{R}\times \mathbb{Z}$ with $\mathbf y_i \cdot \mathbf e_1 \leq \mathbf y_{i+1} \cdot \mathbf e_1$. Then, the following theorem holds. 
\begin{theorem}\label{semi}
The Busemann increments from the Brownian LPP and the O'Connell-Yor polymer, $\big\{B^{\theta_i}_{\mathbf y_{i-1}, \mathbf y_i}\big\}_{i=1}^n$, are mutually independent. 
\end{theorem}

Lastly, we state our result for the two continuous models: the KPZ equation and the directed landscape. From the seminar work of Kardar, Parisi, and Zhang \cite{Kar-Par-Zha-86}, the KPZ equation with initial condition $h_s$ (which is at time $s$) is defined for $t\geq s$ as
$$
\partial_t h(x,t) = \frac{1}{2}\partial_{xx}h(x,t) + \frac{1}{2}(\partial_xh(x,t))^2 + \mathcal{W}(x,t), \quad h(x,s) = h_s(x)
$$
where $\mathcal{W}$ is a space-time white noise. As defined, this equation is ill-posed, but formally, the KPZ equation can be solved via the Cole-Hopf transformation $h(x,t) = \log Z(x,t)$, where $Z$ is the solution to the stochastic heat equation with multiplicative noise:
\begin{equation}\label{she}
\partial_tZ(x,t) = \frac{1}{2}\partial_{xx} Z(x,t) + Z(x,t) \mathcal{W}(x,t), \quad Z(x,s) = e^{h_s(x)}.
\end{equation}
The works \cite{alberts2022greens, MR3162542, MR3189070} constructs a strictly positive, continuous four-parameter random field $\{Z(y,t|x, s) : x,y\in \mathbb{R}, t>s\}$ on the probability space of the white noise $W$ so that 
$$(y,t) \mapsto \int_{\mathbb{R}}e^{h_0(x)}Z(y,t|x,s) dx$$
solves the SHE \eqref{she}. 
The Busemann functions for the KPZ equation or the SHE are constructed in \cite{janjigian2022ergodicity}. For fixed $\theta\in \mathbb{R}$, 
$$B^\theta(x,s;y,t) =  \lim_{r\to -\infty} \log Z(x,s|z_r, r) - \log Z(y,t|z_r,r)$$
where $\{z_r: r<s\wedge t\}$ is any collection of real numbers that satisfies $\lim_{r\to-\infty} \frac{z_r}{|r|} = \theta$.

The directed landscape is a random continuous field landscape is a random continuous function $\mathcal{L}(x,s; y,t)$ where $s,x,t,y \in \mathbb{R}$ with $s<t$. This was constructed in the seminar work \cite{KPZ_DL} and it is conjectured to be the scaling limit of a large class of models in the KPZ universality. The Busemann function (in a fixed direction) for the directed landscape was constructed in the work \cite{rahman2022infinite}. For fixed $\theta\in \mathbb{R}$, 
$$B^\theta(x,s;y,t) = \lim_{r\to\infty} \mathcal{L}(x,s; z_r,r) -  \mathcal{L}(y,t; z_r, r)$$
where $\{z_r: r>s\wedge t\}$ is any collection of real numbers that satisfies $\lim_{r\to \infty} \frac{z_r}{r} = \theta$.

 Let us fix a sequence of directions $\theta_1 \geq \theta_2 \geq \dots  \geq \theta_n$, and a sequence of points $\mathbf y_0, \dots, \mathbf y_n$ along a horizontal line in $\mathbb{R}^2$ with $\mathbf y_i \cdot \mathbf e_1 \leq \mathbf y_{i+1} \cdot \mathbf e_1$. Then, the following theorem holds. 
\begin{theorem}\label{cont}
The Busemann increments from the KPZ equation and the directed landscape, $\big\{B^{\theta_i}_{\mathbf y_{i-1}, \mathbf y_i}\big\}_{i=1}^n$, are mutually independent. 
\end{theorem}
This follows from Theorem \ref{semi}, since it is known that along a horizontal line, the finite-dimensional distributions of the Busemann process from the KPZ equation and the directed landscape are equal to the ones from the O'Connell-Yor polymer and Brownian LPP, respectively, after suitable reflection and scaling \cite{groathouse2023jointly, 2021arXiv211210729S}.


\appendix
\section{Appendix: Random walk estimates}
\subsection{Running maximum upper bound}\label{sub_exp_sec}
First, we state a general result for the running maximum of sub-exponential random variables.
Recall that a random variable $X_1$ is sub-exponential if there exist two positive constants $K_0$ and $\lambda_0$ such that 
\be\label{sub_exp}
\log(\mathbb{E}[e^{\lambda (X_1-\mathbb{E}[X_1])}]) \leq K_0 \lambda^2 \quad \textup{ for $\lambda \in [0, \lambda_0]$}.
\ee
Let $\{X_i\}$ be a sequence of i.i.d.~sub-exponential random variables with the parameters $K_0$ and $\lambda_0$. Define $S_0 = 0$ and $S_k = X_1 + {\mydots} + X_k - k\mathbb{E}[X_1]$ for $k\geq 1$. The following theorem gives an upper bound for the right tail of the running maximum.  
\begin{theorem}[{\cite[Theorem A.11]{ras-sep-she-}}]\label{max_sub_exp}
Let the random walk $S_k$ be defined as above. Then,
$$\mathbb{P} \Big(\max_{0\leq k \leq n} S_k \geq t\sqrt{n}\Big) \leq 
\begin{cases}
e^{-t^2/(4K_0)} \quad  & \textup{if $t \leq 2\lambda_0 K_0 \sqrt n$} \\
e^{-\frac{1}{2}\lambda_0 t\sqrt{n}} \quad  & \textup{if $t \geq 2\lambda_0 K_0 \sqrt n$}
\end{cases}.
$$ 
\end{theorem}

Next, we verify that log gamma and log inverse-gamma random variables are sub-exponential. Recall that if $X\sim\textup{Ga}(\alpha)$, then $\mathbb E[\log X]=\Psi_0(\alpha)$, where $\Psi_0$ is the digamma function, i.e.\ $\Psi_0(\alpha)=(\log\Gamma(\alpha))'$.

\begin{proposition}[{\cite[Proposition  A.12]{ras-sep-she-}}]\label{Ga_sub_exp}
Fix $\varepsilon \in (0, \mu/2)$. There exist positive constants $K_0, \lambda_0$ depending on $\varepsilon$ such that for each $\alpha \in [\varepsilon, \mu-\varepsilon]$ and $X\sim \textup{Ga}(\alpha)$, we have
$$
\mathbb{E}[e^{\lambda(\log X - \Psi_0(\alpha))}]\leq e^{K_0 \lambda^2} \qquad \textup{ for all $\lambda \in [-\lambda_0, \lambda_0]$}.
$$
\end{proposition}

\subsection{Running maximum lower bound}\label{sec_rw}

The following proposition states that even if the random walk has a small negative drift, it is unlikely that the running maximum is very small. 
Let $\{X_i\}_{i\in \mathbb{Z}_{>0}}$ be an i.i.d.~sequence of random variables with
$$\mathbb{E}[X_i] = \mu, \quad \Var[X_i] = 1 \quad \text{and} \quad \mathbb{E}[|X_i-\mu|^3] = c_3 < \infty.$$ Define $S_k = \sum_{i=1}^k X_i$ for $k \geq 1$. 
We have the following proposition, which bounds the probability that the running maximum of a random walk is small. 
\begin{proposition}[{\cite[Proposition A.13]{ras-sep-she-}}]\label{rwest}  There exists a positive  constant $C$ such that for any $l >0$, we have
\beq\mathbb{P}\Big(\max_{1\leq k \leq N} S_k < l\Big) \leq C(c_3l+c_3^2)(|\mu| + 1/\sqrt N).\eeq
\end{proposition}

\bibliographystyle{amsplain}
\bibliography{time}

\end{document}